\documentclass[]{interact}

\usepackage{epstopdf}% To incorporate .eps illustrations using PDFLaTeX, etc.
\usepackage[caption=false]{subfig}% Support for small, `sub' figures and tables

\usepackage[numbers,sort&compress]{natbib}% Citation support using natbib.sty
\bibpunct[, ]{[}{]}{,}{n}{,}{,}% Citation support using natbib.sty
\renewcommand\bibfont{\fontsize{10}{12}\selectfont}% Bibliography support using natbib.sty

\usepackage{mathtools}
\usepackage{bm}
\usepackage{placeins}
\usepackage{algorithm}
\usepackage{algpseudocode}
\usepackage[hidelinks]{hyperref}
\newcommand{\R}{\mathbb{R}}

\newcommand{\ip}[2]{\left\langle #1,#2\right\rangle}
\newcommand{\norm}[1]{\left\lVert #1\right\rVert}
\DeclareMathOperator{\prox}{prox}

\DeclareMathOperator{\argmin}{argmin}

\theoremstyle{plain}
\newtheorem{theorem}{Theorem}
\newtheorem{lemma}{Lemma}
\newtheorem{proposition}{Proposition}

\theoremstyle{definition}
\newtheorem{assumption}{Assumption}

\theoremstyle{remark}
\newtheorem{remark}{Remark}

\begin{document}

\articletype{RESEARCH ARTICLE}

\title{Prox-NAG-GS: A Semi-Implicit Proximal Method for Composite Optimization}

\author{
\name{Sikeh Gisele Wiykiynyuy\textsuperscript{a}, Kelvin Asu Ekuri\textsuperscript{a} and Valentin Leplat\textsuperscript{a}}
\affil{\textsuperscript{a}Institute of Data Science and Artificial Intelligence, Innopolis University}
}

\maketitle

\begin{abstract}
Composite optimization problems, where a smooth loss is combined with a nonsmooth regularizer,
are common in machine learning and inverse problems. In this work, we study a proximal extension
of NAG-GS, a semi-implicit accelerated method obtained from a Gauss-Seidel discretization of an
inertial dynamics. The proposed method, called Prox-NAG-GS, keeps the coupled structure of
NAG-GS for the smooth part and replaces the second update by a proximal step. It therefore
applies to objectives of the form \(F=f+r\), where \(f\) is smooth and \(r\) is convex and
proximable. We derive deterministic convergence guarantees for this method. The analysis has to account for
a specific feature of the scheme. Prox-NAG-GS keeps two coupled sequences: an \(x\)-sequence,
on which the gradient of the smooth term is evaluated, and a \(v\)-sequence, produced by the
proximal update. The gradient is evaluated at \(x_{k+1}\), whereas the proximal step returns
\(v_{k+1}\), which creates a mismatch absent from the standard proximal-gradient analysis. Under
the sufficient condition that the proximal quadratic parameter is at least as large as the smoothness
constant of \(f\), we control this mismatch through an augmented Lyapunov function involving
both sequences. This gives a linear convergence result in the strongly convex composite case. In
the convex case, the same Lyapunov structure yields an \(O(1/k)\) rate for the best iterate and
for the averaged iterate. We test the method on deterministic Elastic Net and Group Lasso problems, and on stochastic
sparse softmax-regression benchmarks. In the deterministic tests, Prox-NAG-GS reaches the same
solutions as the baselines with substantially fewer iterations; for Group Lasso this also gives the
best wall-clock time. In the stochastic tests, Prox-NAG-GS compares favorably with Prox-SGD in
terms of data-fit reduction and gives similar test accuracies. The results also reveal a different
regularization behavior: Prox-SGD usually produces sparser models and may give a lower full
regularized objective when the nonsmooth regularization is large.
\end{abstract}

\begin{keywords}
Composite optimization; proximal algorithms; semi-implicit methods; Lyapunov analysis; nonsmooth convex optimization; stochastic composite optimization
\end{keywords}

\section{Introduction}

Many optimization problems in machine learning, signal processing and inverse problems combine
a smooth loss with a nonsmooth regularizer or constraint. A typical form is
\begin{equation}
\label{eq:intro-composite}
    \min_{x\in\R^d} F(x) := f(x)+r(x),
\end{equation}
where \(f\) is differentiable and \(r\) is convex, possibly nonsmooth, but has an efficient proximal
operator. This setting appears for instance in sparse regression, sparse logistic or softmax
regression, Group Lasso, and constrained learning problems. In these examples, differentiating
the regularizer is either impossible or not desirable. The usual approach is therefore to treat the
smooth term by a gradient step and the nonsmooth term by a proximal step; see, e.g.,
\cite{CombettesPesquet2011,ParikhBoyd2014}.

Proximal-gradient methods and their accelerated variants are among the standard algorithms for
such problems. They are simple, robust, and well understood; classical examples include ISTA and
FISTA~\cite{BeckTeboulle2009}, building on the acceleration ideas of Nesterov~\cite{Nesterov1983}.
At the same time, accelerated and inertial methods can be sensitive to the way the underlying
dynamics is discretized. This is one of the motivations behind NAG-GS~\cite{LeplatNAGGS2022},
a semi-implicit method obtained from a Gauss-Seidel discretization of an accelerated
continuous-time model. The method keeps two coupled variables and evaluates the gradient at
the new \(x\)-iterate. This gives a different stability behavior compared with fully explicit
momentum schemes.

The original NAG-GS method was designed for smooth objectives, or for settings where one can
use a gradient-like oracle for the whole objective. In composite optimization, this is not the natural
situation. Applying a gradient-type update directly to \(F=f+r\) would ignore the structure of the
nonsmooth term. Instead, the regularizer should be handled through its proximal operator. This
motivates the question studied in this paper: can one combine the semi-implicit structure of
NAG-GS with a proximal step for the nonsmooth part?

The construction is direct. The second update of NAG-GS can be interpreted as the minimizer of
a quadratic model. Adding the nonsmooth term \(r\) to this quadratic model leads to a proximal
update. The resulting method, which we call Prox-NAG-GS, reduces to NAG-GS when \(r=0\),
and can be implemented with the same proximal maps used in classical splitting methods.

Our contributions are as follows.
\begin{enumerate}
    \item[(i)] We introduce Prox-NAG-GS, a proximal extension of NAG-GS~\cite{LeplatNAGGS2022}
    for composite objectives of the form~\eqref{eq:intro-composite}.

    \item[(ii)] We prove deterministic convergence results for the proposed method. In the strongly
    convex composite case, we obtain a linear convergence rate under the conservative condition
    that the proximal quadratic parameter is at least as large as the smoothness constant of \(f\).
    The proof relies on an augmented Lyapunov function involving both \(v_k\) and \(x_k\).
    We also show that the same Lyapunov structure gives an \(O(1/k)\) guarantee for the best
    iterate and for the averaged iterate in the convex case.

    \item[(iii)] We evaluate the method on deterministic Elastic Net and Group Lasso benchmarks, and
    compare it with ISTA, FISTA, and Chambolle-Pock.

    \item[(iv)] We also test stochastic sparse softmax regression with entrywise \(\ell_1\) and Group
    Lasso penalties. These experiments show that Prox-NAG-GS compares favorably with Prox-SGD in terms of
data-fit reduction and gives similar test accuracies, while Prox-SGD usually produces sparser
solutions.
\end{enumerate}

The remainder of the paper is organized as follows. 
Section~\ref{sec:related-work} reviews the most relevant proximal, inertial and accelerated methods for composite optimization. 
Section~\ref{sec:problem-algorithm} introduces the composite problem and derives Prox-NAG-GS from the semi-implicit structure of NAG-GS. 
Section~\ref{sec:theory} presents the deterministic convergence analysis. We first establish the key one-step inequalities, then prove a linear convergence result in the strongly convex case, and finally derive an \(O(1/k)\) guarantee in the convex case. 
Section~\ref{sec:numerics} reports numerical experiments on deterministic and stochastic composite optimization problems, with comparisons to standard proximal methods. Finally, Section~\ref{sec:conclusion} concludes and outlines future directions.

\subsection{Related work}
\label{sec:related-work}

\paragraph{Proximal-gradient and splitting methods.}
Problem~\eqref{eq:intro-composite} is the standard setting of proximal-gradient, or
forward-backward, methods. These methods treat the smooth term \(f\) explicitly by a gradient
step and the nonsmooth term \(r\) implicitly through its proximal operator. In the case of a
least-squares loss with an \(\ell_1\) penalty, this gives ISTA~\cite{DaubechiesDefriseDeMol2004}.
More generally, forward-backward splitting is a basic tool for convex composite optimization;
see, e.g., \cite{CombettesWajs2005,CombettesPesquet2011,ParikhBoyd2014,BauschkeCombettes2017}.
Accelerated variants, in particular FISTA~\cite{BeckTeboulle2009}, are standard baselines for
deterministic composite optimization. Related accelerated proximal-gradient schemes were also
studied by Tseng~\cite{Tseng2008} and Nesterov~\cite{Nesterov2013Composite}.

Other splitting methods are also relevant. Primal-dual methods, such as Chambolle-Pock
\cite{ChambollePock2011} and Condat-V\~u type algorithms~\cite{Condat2013,Vu2013}, are
especially useful when the nonsmooth term is composed with a linear operator or when constraints
are better handled through a saddle-point formulation. In our numerical experiments, we include
Chambolle-Pock as one of the deterministic baselines.

\paragraph{Stochastic composite optimization.}
In large-scale learning, the smooth term often has the finite-sum form
\[
    f(x)=\frac1n\sum_{i=1}^n f_i(x),
\]
and gradients are computed from mini-batches. This has motivated stochastic proximal and
stochastic composite-gradient methods. Early examples include FOBOS~\cite{DuchiSinger2009}
and regularized dual averaging~\cite{Xiao2010}. Stochastic approximation methods for composite
objectives were further developed in, for example, \cite{Lan2012,GhadimiLan2013,GhadimiLanZhang2016};
see also the survey~\cite{BottouCurtisNocedal2018} for a broader discussion of stochastic
optimization.

For finite-sum problems, variance-reduced methods such as SVRG, Prox-SVRG, SAG, and SAGA
are also important references~\cite{JohnsonZhang2013,XiaoZhang2014,SchmidtLeRouxBach2017,DefazioBachLacosteJulien2014}.
They are not the main focus of this paper. In the stochastic experiments, we use Prox-SGD as a
first baseline, because it is the direct stochastic counterpart of proximal-gradient descent and uses
the proximal operator at every stochastic step.

\paragraph{Acceleration, inertial dynamics, and NAG-GS.}
Acceleration can be studied from the discrete viewpoint, as in Nesterov's accelerated gradient
method~\cite{Nesterov1983,Nesterov2018}, or through continuous-time inertial dynamics. The
continuous-time point of view has led to a better understanding of accelerated methods and their
Lyapunov functions; see, e.g., \cite{SuBoydCandes2016,AttouchChbaniRiahi2019,WilsonRechtJordan2021}.
There is also a large literature on inertial proximal and inertial forward-backward methods, going
back to the heavy-ball method~\cite{Polyak1964} and including splitting schemes such as
\cite{AlvarezAttouch2001,MoudafiOliny2003,OchsChenBroxPock2014}. %AttouchMaingeRedont2014 ?

NAG-GS~\cite{LeplatNAGGS2022} belongs to this line of work, but uses a semi-implicit
Gauss-Seidel discretization of an accelerated dynamics. The present paper studies a proximal
extension of this idea. The resulting method is close in spirit to forward-backward splitting, but
the NAG-GS coupling changes the analysis: the gradient is evaluated at \(x_{k+1}\), while the
proximal step produces \(v_{k+1}\). This mismatch is the main point handled in the convergence
proof.

\section{Composite problem and Prox-NAG-GS}
\label{sec:problem-algorithm}

\subsection{Problem class}

We consider the composite optimization problem
\begin{equation}
\label{eq:composite-problem}
    \min_{x\in\R^d} F(x) := f(x)+r(x),
\end{equation}
where \(f:\R^d\to\R\) is differentiable, and
\(r:\R^d\to\R\cup\{+\infty\}\) is proper, closed and convex. We assume that
\(r\) is proximable, in the sense that, for every \(\lambda>0\), the proximal
operator
\begin{equation}
\label{eq:prox-def}
    \prox_{\lambda r}(z)
    :=
    \argmin_{u\in\R^d}
    \left\{
        r(u)+\frac{1}{2\lambda}\norm{u-z}^2
    \right\}
\end{equation}
can be computed efficiently. This framework also covers simple convex
constraints, by taking \(r\) to be the indicator function of a closed convex
set, in which case the proximal operator is the Euclidean projection.

Two examples will be used in the numerical section. For the \(\ell_1\) penalty
\(r(x)=\lambda_1\norm{x}_1\), the proximal operator is the soft-thresholding
map. For the Group Lasso penalty
\[
    r(x)=\lambda_g\sum_{G\in\mathcal G}\norm{x_G}_2,
\]
the proximal operator is block separable and, on each group, is given by
\begin{equation}
\label{eq:group-prox}
    \prox_{\tau\lambda_g\norm{\cdot}_2}(x_G)
    =
    \left(1-\frac{\tau\lambda_g}{\norm{x_G}_2}\right)_+ x_G,
\end{equation}
with the usual convention that the right-hand side is zero when \(x_G=0\).

For the deterministic analysis, we will use the following assumptions. The
first one corresponds to the convex case, while the second one will be added
for the linear convergence result.

\begin{assumption}[Convex composite setting]
\label{ass:deterministic-convex}
The function \(f\) is convex, differentiable and \(L\)-smooth. The function
\(r\) is proper, closed and convex. The composite function \(F=f+r\) admits a
minimizer \(x^\star\).
\end{assumption}

\begin{assumption}[Strongly convex composite setting]
\label{ass:deterministic-strongly-convex}
Assumption~\ref{ass:deterministic-convex} holds. Moreover, \(f\) is
\(\mu_f\)-strongly convex, and \(F=f+r\) is \(\mu_F\)-strongly convex.
\end{assumption}

Since \(r\) is convex, one can take \(\mu_F=\mu_f\) when \(f\) is
\(\mu_f\)-strongly convex. We keep the two constants in the notation because
the proof only uses strong convexity of \(f\) in one place, and strong
convexity of \(F\) in another.

We also distinguish these curvature constants from the algorithmic parameter
\(\widehat\mu\) used in Prox-NAG-GS. This distinction is important.
The parameter \(\widehat\mu\) controls the quadratic term in the proximal
subproblem. In the convergence analysis, we will impose the sufficient condition
\(\widehat\mu\ge L\). This condition is conservative, but it gives a regime in
which the proximal quadratic model is strong enough to control the mismatch
created by the semi-implicit update. In the numerical experiments,
\(\widehat\mu\) is treated as a tunable parameter.

\subsection{From NAG-GS to a proximal update}

The smooth NAG-GS method uses two variables, \(x_k\) and \(v_k\). Given a step
parameter \(\alpha_k>0\), we define
\[
    a_k=\frac{\alpha_k}{1+\alpha_k}.
\]
The first update is
\begin{equation}
\label{eq:x-update}
    x_{k+1}=(1-a_k)x_k+a_kv_k.
\end{equation}
The second update can be written as a gradient step from an intermediate point.
Let
\[
    z_{k+1}:=(1-b_k)v_k+b_kx_{k+1},
\]
where \(b_k\in(0,1)\) is defined below. In the smooth case, the update
\[
    v_{k+1}
    =
    z_{k+1}-\frac{b_k}{\widehat\mu}g_{k+1}
\]
is equivalently the minimizer of the quadratic model
\[
    v
    \mapsto
    \ip{g_{k+1}}{v}
    +
    \frac{\widehat\mu}{2b_k}\norm{v-z_{k+1}}^2.
\]
Here \(g_{k+1}\) is a gradient, or a gradient estimator, of the smooth part at
\(x_{k+1}\).

This form makes the proximal extension immediate. We add the nonsmooth term
\(r(v)\) to the quadratic model. This gives
\begin{equation}
\label{eq:prox-naggs-vupdate}
    v_{k+1}
    =
    \prox_{\frac{b_k}{\widehat\mu}r}
    \left(
        z_{k+1}-\frac{b_k}{\widehat\mu}g_{k+1}
    \right).
\end{equation}
When \(r=0\), this reduces to the smooth NAG-GS update.

\subsection{The algorithm}

The resulting method is given in Algorithm~\ref{alg:prox-naggs}. In the
deterministic case, we use \(g_{k+1}=\nabla f(x_{k+1})\). In the stochastic
case, \(g_{k+1}\) is computed from a mini-batch. The convergence analysis in
Section~\ref{sec:theory} is deterministic; the stochastic variant is studied
numerically in Section~\ref{sec:numerics}.

\begin{algorithm}[t]
\caption{Prox-NAG-GS for composite optimization}
\label{alg:prox-naggs}
\begin{algorithmic}[1]
\Require \(x_0\in\R^d\), \(v_0=x_0\), parameters \(\widehat\mu>0\),
\(\gamma_0>0\), step parameters \(\alpha_k>0\), proximal operator
\(\prox_{\lambda r}\)
\For{\(k=0,1,2,\ldots\)}
    \State \(a_k \gets \alpha_k/(1+\alpha_k)\)
    \State \(x_{k+1}\gets (1-a_k)x_k+a_kv_k\)
    \State \(b_k \gets \alpha_k\widehat\mu/(\alpha_k\widehat\mu+\gamma_k)\)
    \State \(z_{k+1}\gets (1-b_k)v_k+b_kx_{k+1}\)
    \State compute \(g_{k+1}\), a gradient or stochastic gradient of \(f\) at \(x_{k+1}\)
    \State
    \[
        v_{k+1}
        \gets
        \prox_{\frac{b_k}{\widehat\mu}r}
        \left(
            z_{k+1}-\frac{b_k}{\widehat\mu}g_{k+1}
        \right)
    \]
    \State \(\gamma_{k+1}\gets (1-a_k)\gamma_k+a_k\widehat\mu\)
\EndFor
\end{algorithmic}
\end{algorithm}

The method is close in spirit to forward-backward splitting, since the smooth
part is treated by a gradient step and the nonsmooth part by a proximal step.
The difference is that the gradient is evaluated at the new point \(x_{k+1}\),
which is coupled with \(v_k\), while the proximal step produces \(v_{k+1}\).
This is the semi-implicit structure inherited from NAG-GS. It is also the main
point that has to be handled in the convergence analysis.

\section{Convergence analysis}
\label{sec:theory}

In this section, we study Prox-NAG-GS in the deterministic composite setting. The analysis has
two parts. We first derive estimates that are valid under the convex composite setting of
Assumption~\ref{ass:deterministic-convex}. We then use these estimates in two different ways.
Under the strongly convex setting of Assumption~\ref{ass:deterministic-strongly-convex}, we
prove a linear convergence result. In the merely convex case, we prove an \(O(1/k)\) guarantee for
the best iterate and for the averaged iterate.

The analysis is not a direct application of the standard proximal-gradient proof. The difficulty
comes from the point at which the gradient is evaluated. The proximal step returns \(v_{k+1}\),
while the gradient of the smooth part is evaluated at \(x_{k+1}\). Hence the optimality condition
of the proximal subproblem does not directly give an element of \(\partial F(v_{k+1})\). This
produces a mismatch term that has to be controlled.

The proof is organized as follows. We first fix a constant-parameter regime and derive the
proximal identities, the geometric identities, and the main one-step inequality. These estimates
are common to both convergence results. The mismatch term is then absorbed under the sufficient
condition \(\widehat\mu\ge L\). For the strongly convex case, a natural Lyapunov candidate does
not close, and one has to add a term involving \(\|x_k-x^\star\|^2\). This gives a linear
convergence result. Finally, we show that the same augmented energy also gives a descent
inequality in the convex case.

Throughout this section, \(x^\star\) denotes a minimizer of \(F=f+r\), and
\[
    F^\star := F(x^\star).
\]
Unless stated otherwise, Assumption~\ref{ass:deterministic-convex} is in force. The additional
strong convexity assumptions are only used in the subsection devoted to the linear convergence
result.

\subsection{Constant-parameter regime}

We focus on the constant-damping regime
\[
    \gamma_0=\widehat\mu.
\]
Then the recursion for \(\gamma_k\) gives \(\gamma_k\equiv\widehat\mu\), and therefore
\[
    b_k=a_k=\frac{\alpha_k}{1+\alpha_k}.
\]
% In the linear-rate result, we further assume that \(a_k\) is constant:
% \[
%     a_k \equiv a\in(0,1).
% \]
In the convergence results below, we further assume that \(a_k\) is constant:
\[
    a_k\equiv a\in(0,1).
\]
The deterministic Prox-NAG-GS iteration then reads
\begin{equation}
\label{eq:constant-prox-naggs-theory}
\begin{aligned}
    x_{k+1}&=(1-a)x_k+av_k,\\
    z_{k+1}&=(1-a)v_k+ax_{k+1},\\
    v_{k+1}
    &=
    \prox_{\frac{a}{\widehat\mu}r}
    \left(z_{k+1}-\frac{a}{\widehat\mu}\nabla f(x_{k+1})\right).
\end{aligned}
\end{equation}
The parameter \(\widehat\mu\) is an algorithmic curvature parameter. It is not necessarily equal
to the strong convexity constant of \(f\). In the proof below, we will require \(\widehat\mu\ge L\).
This is a sufficient condition ensuring that the quadratic model used in the proximal step is strong
enough to compensate for the smoothness error of \(f\).

\subsection{Proximal identities}

We first collect two standard consequences of the proximal update.

\begin{lemma}[Optimality condition]
\label{lem:prox-opt-new}
Let \(v_{k+1}\) be generated by~\eqref{eq:constant-prox-naggs-theory}. Then \(v_{k+1}\) is the
unique minimizer of
\[
    v\mapsto
    r(v)
    +\ip{\nabla f(x_{k+1})}{v}
    +\frac{\widehat\mu}{2a}\norm{v-z_{k+1}}^2.
\]
Moreover, there exists \(s_{k+1}\in\partial r(v_{k+1})\) such that
\begin{equation}
\label{eq:prox-opt-new}
    0
    =
    s_{k+1}
    +\nabla f(x_{k+1})
    +\frac{\widehat\mu}{a}(v_{k+1}-z_{k+1}).
\end{equation}
Equivalently, if
\begin{equation}
\label{eq:q-def-new}
    q_{k+1}:=\frac{\widehat\mu}{a}(z_{k+1}-v_{k+1}),
\end{equation}
then
\begin{equation}
\label{eq:q-gradient-subgrad-new}
    q_{k+1}=\nabla f(x_{k+1})+s_{k+1}.
\end{equation}
\end{lemma}

\begin{proof}
By definition of the proximal operator, \(v_{k+1}\) minimizes
\[
    r(v)
    +
    \frac{\widehat\mu}{2a}
    \left\|
        v-\left(z_{k+1}-\frac{a}{\widehat\mu}\nabla f(x_{k+1})\right)
    \right\|^2.
\]
Expanding the square and removing the terms independent of \(v\) gives the stated minimization
problem. Since \(r\) is convex and the quadratic term is strongly convex, this minimizer is unique. The optimality condition
gives~\eqref{eq:prox-opt-new}, and~\eqref{eq:q-gradient-subgrad-new} follows by rearranging the
terms.
\end{proof}

\begin{lemma}[Three-point inequality]
\label{lem:three-point-new}
For every \(u\in\R^d\), the iterate \(v_{k+1}\) satisfies
\begin{align}
\label{eq:three-point-new}
    &r(v_{k+1})
    +\ip{\nabla f(x_{k+1})}{v_{k+1}}
    +\frac{\widehat\mu}{2a}\norm{v_{k+1}-z_{k+1}}^2
    \nonumber\\
    &\qquad \le
    r(u)
    +\ip{\nabla f(x_{k+1})}{u}
    +\frac{\widehat\mu}{2a}\norm{u-z_{k+1}}^2
    -\frac{\widehat\mu}{2a}\norm{u-v_{k+1}}^2 .
\end{align}
\end{lemma}

\begin{proof}
The function minimized in Lemma~\ref{lem:prox-opt-new} is \(\widehat\mu/a\)-strongly convex and
is minimized at \(v_{k+1}\). Therefore, for every \(u\),
\[
    \phi_k(u)
    \ge
    \phi_k(v_{k+1})
    +\frac{\widehat\mu}{2a}\norm{u-v_{k+1}}^2,
\]
where
\[
    \phi_k(v)
    =
    r(v)
    +\ip{\nabla f(x_{k+1})}{v}
    +\frac{\widehat\mu}{2a}\norm{v-z_{k+1}}^2.
\]
Rearranging gives~\eqref{eq:three-point-new}.
\end{proof}

The next observation explains where the proof differs from the standard proximal-gradient
analysis. By Lemma~\ref{lem:prox-opt-new},
\[
    q_{k+1}
    =
    \nabla f(x_{k+1})+s_{k+1},
    \qquad s_{k+1}\in\partial r(v_{k+1}).
\]
If the gradient were evaluated at \(v_{k+1}\), then \(q_{k+1}\) would be an element of
\(\partial F(v_{k+1})\). Here, instead,
\[
    q_{k+1}
    =
    \underbrace{\nabla f(v_{k+1})+s_{k+1}}_{\in\partial F(v_{k+1})}
    +
    \underbrace{\nabla f(x_{k+1})-\nabla f(v_{k+1})}_{\text{mismatch}}.
\]
The \(L\)-smoothness of \(f\) gives
\[
    \norm{\nabla f(x_{k+1})-\nabla f(v_{k+1})}
    \le
    L\norm{x_{k+1}-v_{k+1}}.
\]
Thus the quantity \(\norm{x_{k+1}-v_{k+1}}^2\) will appear in the proof.

\subsection{Geometric identities}

We now record the identities coming only from the two affine updates in
\eqref{eq:constant-prox-naggs-theory}.

\begin{lemma}[Kinematic identities]
\label{lem:kinematic-new}
For every \(k\),
\begin{equation}
\label{eq:xplus-v-new}
    x_{k+1}-v_k=(1-a)(x_k-v_k),
\end{equation}
\begin{equation}
\label{eq:xplus-z-new}
    x_{k+1}-z_{k+1}=(1-a)^2(x_k-v_k),
\end{equation}
and
\begin{equation}
\label{eq:mismatch-decomp-new}
    x_{k+1}-v_{k+1}
    =
    (1-a)^2(x_k-v_k)+(z_{k+1}-v_{k+1}).
\end{equation}
\end{lemma}

\begin{proof}
The first identity follows directly from
\[
    x_{k+1}=(1-a)x_k+av_k.
\]
Since
\[
    z_{k+1}=(1-a)v_k+ax_{k+1},
\]
we also have
\[
    x_{k+1}-z_{k+1}
    =
    (1-a)(x_{k+1}-v_k),
\]
which gives~\eqref{eq:xplus-z-new}. The last identity follows by writing
\[
    x_{k+1}-v_{k+1}
    =
    (x_{k+1}-z_{k+1})+(z_{k+1}-v_{k+1}).
\]
\end{proof}

We shall also use the following identity for the distance of \(x_k\) to the solution.

\begin{lemma}[Recursion for \(x_k\)]
\label{lem:x-recursion-new}
Let
\[
    X_k:=\norm{x_k-x^\star}^2,
    \qquad
    V_k:=\norm{v_k-x^\star}^2,
    \qquad
    D_k:=\norm{x_k-v_k}^2.
\]
Then
\begin{equation}
\label{eq:X-recursion-new}
    X_{k+1}
    =
    (1-a)X_k+aV_k-a(1-a)D_k.
\end{equation}
\end{lemma}

\begin{proof}
Since \(x_{k+1}=(1-a)x_k+av_k\), we have
\[
    x_{k+1}-x^\star
    =
    (1-a)(x_k-x^\star)+a(v_k-x^\star).
\]
Using the identity
\[
    \norm{(1-a)p+aq}^2
    =
    (1-a)\norm{p}^2+a\norm{q}^2-a(1-a)\norm{p-q}^2
\]
gives~\eqref{eq:X-recursion-new}.
\end{proof}

\subsection{A one-step inequality}

We now combine the three-point inequality with the smoothness of \(f\). 
The estimate below is written with a parameter \(\mu_f\ge0\) such that \(f\) is
\(\mu_f\)-strongly convex. In the merely convex case, we simply take \(\mu_f=0\). 
When the strongly convex setting is considered, \(\mu_f\) is the strong convexity
constant of \(f\).
\begin{proposition}[One-step inequality]
\label{prop:fundamental-new}
Let \((x_{k+1},z_{k+1},v_{k+1})\) be generated by
\eqref{eq:constant-prox-naggs-theory}. Assume that \(f\) is \(L\)-smooth and
\(\mu_f\)-strongly convex, with the convention \(\mu_f=0\) in the merely convex
case. Then, for every \(u\in\R^d\),
\begin{align}
\label{eq:fundamental-new}
    F(v_{k+1})-F(u)
    &\le
    \frac{\widehat\mu}{2a}
    \left(
        \norm{u-z_{k+1}}^2
        -\norm{u-v_{k+1}}^2
        -\norm{v_{k+1}-z_{k+1}}^2
    \right)
    \nonumber\\
    &\quad
    -\frac{\mu_f}{2}\norm{u-x_{k+1}}^2
    +\frac{L}{2}\norm{v_{k+1}-x_{k+1}}^2.
\end{align}
\end{proposition}

\begin{proof}
By \(L\)-smoothness of \(f\),
\[
    f(v_{k+1})
    \le
    f(x_{k+1})
    +\ip{\nabla f(x_{k+1})}{v_{k+1}-x_{k+1}}
    +\frac{L}{2}\norm{v_{k+1}-x_{k+1}}^2.
\]
Adding this inequality to~\eqref{eq:three-point-new} gives
\begin{align*}
    F(v_{k+1})
    &\le
    f(x_{k+1})
    +\ip{\nabla f(x_{k+1})}{u-x_{k+1}}
    +r(u)
    \\
    &\quad
    +\frac{\widehat\mu}{2a}\norm{u-z_{k+1}}^2
    -\frac{\widehat\mu}{2a}\norm{u-v_{k+1}}^2
    -\frac{\widehat\mu}{2a}\norm{v_{k+1}-z_{k+1}}^2
    \\
    &\quad
    +\frac{L}{2}\norm{v_{k+1}-x_{k+1}}^2.
\end{align*}
Using the \(\mu_f\)-strong convexity of \(f\),
\[
    f(x_{k+1})
    +\ip{\nabla f(x_{k+1})}{u-x_{k+1}}
    \le
    f(u)-\frac{\mu_f}{2}\norm{u-x_{k+1}}^2.
\]
Substitution gives the result.
\end{proof}

We next apply Proposition~\ref{prop:fundamental-new} with \(u=x^\star\). The only point to
expand is the term \(\norm{x^\star-z_{k+1}}^2\).

\begin{lemma}[Expansion at the solution]
\label{lem:expanded-new}
Let
\[
    G_k:=F(v_k)-F^\star,
    \qquad
    V_k:=\norm{v_k-x^\star}^2,
    \qquad
    X_k:=\norm{x_k-x^\star}^2,
\]
and
\[
    D_k:=\norm{x_k-v_k}^2,
    \qquad
    R_{k+1}:=\norm{v_{k+1}-z_{k+1}}^2,
    \qquad
    M_{k+1}:=\norm{v_{k+1}-x_{k+1}}^2.
\]
Then
\begin{align}
\label{eq:expanded-new}
    G_{k+1}
    +\frac{\widehat\mu}{2a}V_{k+1}
    &\le
    \frac{\widehat\mu(1-a)}{2a}V_k
    +\frac{\widehat\mu-\mu_f}{2}X_{k+1}
    \nonumber\\
    &\quad
    -\frac{\widehat\mu(1-a)^3}{2}D_k
    -\frac{\widehat\mu}{2a}R_{k+1}
    +\frac{L}{2}M_{k+1}.
\end{align}
\end{lemma}

\begin{proof}
Applying Proposition~\ref{prop:fundamental-new} with \(u=x^\star\) gives
\begin{align*}
    G_{k+1}
    +\frac{\widehat\mu}{2a}V_{k+1}
    &\le
    \frac{\widehat\mu}{2a}\norm{x^\star-z_{k+1}}^2
    -\frac{\mu_f}{2}X_{k+1}
    \\
    &\quad
    -\frac{\widehat\mu}{2a}R_{k+1}
    +\frac{L}{2}M_{k+1}.
\end{align*}
Since
\[
    z_{k+1}=(1-a)v_k+ax_{k+1},
\]
we have
\[
    \norm{x^\star-z_{k+1}}^2
    =
    (1-a)V_k+aX_{k+1}
    -a(1-a)\norm{x_{k+1}-v_k}^2.
\]
By Lemma~\ref{lem:kinematic-new},
\[
    \norm{x_{k+1}-v_k}^2=(1-a)^2D_k.
\]
Substituting these two identities gives~\eqref{eq:expanded-new}.
\end{proof}

\subsection{Why the first Lyapunov candidate does not close}

Let us explain the main difficulty before proving the final contraction. A first natural candidate is
\[
    G_k+\frac{\widehat\mu}{2a}V_k.
\]
This is the expected energy if one tries to imitate a standard strongly convex proximal analysis.
However, Lemma~\ref{lem:expanded-new} contains the additional term
\[
    \frac{L}{2}M_{k+1}
    =
    \frac{L}{2}\norm{v_{k+1}-x_{k+1}}^2.
\]
This term comes from the mismatch between the gradient point \(x_{k+1}\) and the proximal point
\(v_{k+1}\).

If one takes the natural choice \(\widehat\mu=\mu_f\), the term involving \(X_{k+1}\) disappears
from~\eqref{eq:expanded-new}. But the mismatch term remains, and it cannot be absorbed in
general unless the proximal quadratic model is strong enough compared with the smoothness of
\(f\). This is why we impose the sufficient condition
\[
    \widehat\mu\ge L.
\]
This condition allows us to absorb the mismatch term. The price to pay is that, in general,
\(\widehat\mu>\mu_f\), and therefore the positive term
\[
    \frac{\widehat\mu-\mu_f}{2}X_{k+1}
\]
appears in~\eqref{eq:expanded-new}. The identity~\eqref{eq:X-recursion-new} then shows how to
handle it: we add a multiple of \(X_k=\norm{x_k-x^\star}^2\) to the Lyapunov function.

This gives the augmented Lyapunov function
\[
    \mathcal L_k
    =
    G_k
    +bV_k
    +cX_k,
\]
with suitable constants \(b>0\) and \(c>0\).

\subsection{Absorbing the mismatch term}

We now show that \(\widehat\mu\ge L\) is sufficient to remove the mismatch term from the
one-step inequality.

\begin{lemma}[Absorption of the mismatch term]
\label{lem:mismatch-absorption-new}
Assume that \(\widehat\mu\ge L\). Then
\begin{equation}
\label{eq:mismatch-absorption-new}
    -\frac{\widehat\mu(1-a)^3}{2}D_k
    -\frac{\widehat\mu}{2a}R_{k+1}
    +\frac{L}{2}M_{k+1}
    \le 0.
\end{equation}
\end{lemma}

\begin{proof}
By Lemma~\ref{lem:kinematic-new},
\[
    x_{k+1}-v_{k+1}
    =
    (1-a)^2(x_k-v_k)+(z_{k+1}-v_{k+1}).
\]
Using Young's inequality in the form
\[
    \norm{p+q}^2
    \le
    \frac{1}{1-a}\norm{p}^2+\frac{1}{a}\norm{q}^2,
\]
with
\[
    p=(1-a)^2(x_k-v_k),
    \qquad
    q=z_{k+1}-v_{k+1},
\]
we obtain
\[
    M_{k+1}
    \le
    (1-a)^3D_k+\frac1a R_{k+1}.
\]
Therefore,
\begin{align*}
    &-\frac{\widehat\mu(1-a)^3}{2}D_k
    -\frac{\widehat\mu}{2a}R_{k+1}
    +\frac{L}{2}M_{k+1}
    \\
    &\qquad\le
    -\frac{\widehat\mu-L}{2}(1-a)^3D_k
    -\frac{\widehat\mu-L}{2a}R_{k+1}.
\end{align*}
The right-hand side is nonpositive because \(\widehat\mu\ge L\).
\end{proof}

Combining Lemma~\ref{lem:expanded-new} and Lemma~\ref{lem:mismatch-absorption-new} gives
the following simplified estimate.

\begin{proposition}[Reduced one-step inequality]
\label{prop:reduced-one-step-new}
Assume that \(\widehat\mu\ge L\). Then
\begin{equation}
\label{eq:reduced-one-step-new}
    G_{k+1}
    +
    \frac{\widehat\mu}{2a}V_{k+1}
    \le
    \frac{\widehat\mu(1-a)}{2a}V_k
    +
    \frac{\widehat\mu-\mu_f}{2}X_{k+1}.
\end{equation}
\end{proposition}

\begin{proof}
This follows directly from Lemma~\ref{lem:expanded-new} and
Lemma~\ref{lem:mismatch-absorption-new}.
\end{proof}

\subsection{Strongly convex case: augmented Lyapunov function}

We now prove the contraction of the augmented Lyapunov function. Define
\begin{equation}
\label{eq:b-beta-def-new}
    b:=\frac{\widehat\mu}{2a},
    \qquad
    \beta:=\frac{\widehat\mu-\mu_f}{2}.
\end{equation}
With this notation, Proposition~\ref{prop:reduced-one-step-new} becomes
\begin{equation}
\label{eq:reduced-one-step-b-beta-new}
    G_{k+1}+bV_{k+1}
    \le
    b(1-a)V_k+\beta X_{k+1}.
\end{equation}
Let \(c>0\), to be chosen below, and define
\begin{equation}
\label{eq:augmented-lyapunov-new}
    \mathcal L_k
    :=
    G_k+bV_k+cX_k.
\end{equation}

\begin{lemma}[Lyapunov contraction]
\label{lem:lyapunov-contraction-new}
% Assume that \(\widehat\mu\ge L\). Let \(b\) and \(\beta\) be defined by
Assume that Assumption~\ref{ass:deterministic-strongly-convex} holds and that
\(\widehat\mu\ge L\). Let \(b\) and \(\beta\) be defined by
\eqref{eq:b-beta-def-new}. Choose \(c>0\) such that
\begin{equation}
\label{eq:c-interval-new}
    \frac{\beta(1-a)}{a}
    <
    c
    <
    \frac{\widehat\mu+\mu_F}{2a}-\beta .
\end{equation}
Then
\[
    \mathcal L_{k+1}\le \theta \mathcal L_k,
    \qquad \forall k\ge 0,
\]
where
\begin{equation}
\label{eq:theta-def-new}
    \theta
    :=
    \max\left\{
    \frac{b(1-a)+a(\beta+c)}{b+\mu_F/2},
    \frac{(1-a)(\beta+c)}{c}
    \right\}.
\end{equation}
Moreover, \(\theta<1\).
\end{lemma}

\begin{proof}
Adding \(cX_{k+1}\) to~\eqref{eq:reduced-one-step-b-beta-new} gives
\[
    \mathcal L_{k+1}
    \le
    b(1-a)V_k+(\beta+c)X_{k+1}.
\]
Using the identity~\eqref{eq:X-recursion-new}, we get
\[
    \mathcal L_{k+1}
    \le
    b(1-a)V_k
    +(\beta+c)\big((1-a)X_k+aV_k-a(1-a)D_k\big).
\]
% Dropping the nonpositive term \(-a(1-a)(\beta+c)D_k\), we obtain
Since \(\widehat\mu\ge L\ge \mu_f\), we have \(\beta\ge0\). Hence
\(-a(1-a)(\beta+c)D_k\le0\), and dropping this term gives
\begin{equation}
\label{eq:L-upper-new}
    \mathcal L_{k+1}
    \le
    \big[b(1-a)+a(\beta+c)\big]V_k
    +(1-a)(\beta+c)X_k.
\end{equation}

Since \(F\) is \(\mu_F\)-strongly convex and \(x^\star\) minimizes \(F\),
\[
    G_k=F(v_k)-F^\star
    \ge
    \frac{\mu_F}{2}V_k.
\]
Hence
\[
    \mathcal L_k
    =
    G_k+bV_k+cX_k
    \ge
    \left(b+\frac{\mu_F}{2}\right)V_k+cX_k.
\]
Therefore, \(\mathcal L_{k+1}\le\theta\mathcal L_k\) holds with \(\theta\) defined in
\eqref{eq:theta-def-new}.

It remains to check that \(\theta<1\). The second ratio in~\eqref{eq:theta-def-new} is smaller
than one if and only if
\[
    (1-a)(\beta+c)<c,
\]
which is equivalent to
\[
    c>\frac{\beta(1-a)}{a}.
\]
The first ratio is smaller than one if and only if
\[
    b(1-a)+a(\beta+c)<b+\frac{\mu_F}{2}.
\]
Since \(b=\widehat\mu/(2a)\), this is equivalent to
\[
    c<\frac{\widehat\mu+\mu_F}{2a}-\beta.
\]
These are exactly the two inequalities in~\eqref{eq:c-interval-new}.

Finally, the interval in~\eqref{eq:c-interval-new} is nonempty. Indeed,
\[
\begin{aligned}
    \left(\frac{\widehat\mu+\mu_F}{2a}-\beta\right)
    -
    \frac{\beta(1-a)}{a}
    &=
    \frac{1}{a}
    \left(
        \frac{\widehat\mu+\mu_F}{2}
        -\beta
    \right)
    \\
    &=
    \frac{1}{a}
    \left(
        \frac{\widehat\mu+\mu_F}{2}
        -
        \frac{\widehat\mu-\mu_f}{2}
    \right)
    \\
    &=
    \frac{\mu_F+\mu_f}{2a}
    >0.
\end{aligned}
\]
Thus such a \(c\) exists, and for any such choice we have \(\theta<1\).
\end{proof}

We can now state the convergence theorem.

\begin{theorem}[Linear convergence]
\label{thm:linear-conservative-new}
Assume that Assumption~\ref{ass:deterministic-strongly-convex} holds. 
Assume also that Prox-NAG-GS is run in the constant-parameter regime
\[
    x_{k+1}=(1-a)x_k+av_k,
    \qquad
    z_{k+1}=(1-a)v_k+ax_{k+1},
\]
\[
    v_{k+1}
    =
    \prox_{\frac{a}{\widehat\mu}r}
    \left(
        z_{k+1}
        -
        \frac{a}{\widehat\mu}\nabla f(x_{k+1})
    \right),
\]
with \(a\in(0,1)\). Finally, assume that the algorithmic curvature parameter satisfies
\[
    \widehat\mu\ge L.
\]
Let
\[
    b:=\frac{\widehat\mu}{2a},
    \qquad
    \beta:=\frac{\widehat\mu-\mu_f}{2},
\]
and choose any
\[
    c\in
    \left(
        \frac{\beta(1-a)}{a},
        \frac{\widehat\mu+\mu_F}{2a}-\beta
    \right).
\]
Define
\[
    \mathcal L_k
    :=
    F(v_k)-F^\star
    +b\norm{v_k-x^\star}^2
    +c\norm{x_k-x^\star}^2.
\]
Then
\[
    \mathcal L_{k+1}\le \theta \mathcal L_k,
    \qquad \forall k\ge 0,
\]
where
\[
    \theta
    :=
    \max\left\{
    \frac{b(1-a)+a(\beta+c)}{b+\mu_F/2},
    \frac{(1-a)(\beta+c)}{c}
    \right\}
    <1.
\]
Consequently, for every \(k\ge0\),
\[
    F(v_k)-F^\star
    \le
    \mathcal L_0\theta^k,
\]
and
\[
    \norm{v_k-x^\star}^2
    \le
    \frac{\mathcal L_0}{b}\theta^k,
    \qquad
    \norm{x_k-x^\star}^2
    \le
    \frac{\mathcal L_0}{c}\theta^k.
\]
\end{theorem}

\begin{proof}
The contraction
\[
    \mathcal L_{k+1}\le \theta\mathcal L_k
\]
is exactly Lemma~\ref{lem:lyapunov-contraction-new}. Iterating it gives
\[
    \mathcal L_k\le \theta^k\mathcal L_0.
\]
Since
\[
    F(v_k)-F^\star\le \mathcal L_k,
    \qquad
    b\norm{v_k-x^\star}^2\le \mathcal L_k,
    \qquad
    c\norm{x_k-x^\star}^2\le \mathcal L_k,
\]
the three estimates follow.
\end{proof}

\begin{remark}[Role of the condition \(\widehat\mu\ge L\)]
The condition \(\widehat\mu\ge L\) is sufficient, not necessary. It means that the quadratic term
in the proximal subproblem is strong enough to dominate the smoothness error caused by the
mismatch between \(x_{k+1}\) and \(v_{k+1}\). Equivalently, the effective proximal-gradient stepsize
\(a/\widehat\mu\) is at most \(a/L\).

This condition is conservative. In the numerical experiments, the parameters are tuned empirically,
and the best values can be more aggressive than those covered by the theorem. The theorem should
therefore be read as a stability and convergence guarantee for Prox-NAG-GS, not as a complete
description of all parameter regimes that work in practice.
\end{remark}

\begin{remark}[On the rate]
The theorem proves a linear rate, but we do not claim that the contraction factor \(\theta\) is
optimal. For strongly convex composite problems, accelerated proximal-gradient methods can
achieve sharper worst-case rates. The contribution here is different: we prove that the semi-implicit
proximal coupling used by Prox-NAG-GS admits a Lyapunov function and converges linearly in a
well-defined deterministic regime.
\end{remark}

\begin{remark}[Why the additional \(X_k\) term is needed]
The term \(cX_k=c\norm{x_k-x^\star}^2\) is not added for cosmetic reasons. After absorbing the
mismatch term with \(\widehat\mu\ge L\), the one-step inequality still contains the positive term
\(\beta X_{k+1}\), where \(\beta=(\widehat\mu-\mu_f)/2\). The identity
\[
    X_{k+1}
    =
    (1-a)X_k+aV_k-a(1-a)D_k
\]
allows this term to be propagated through the Lyapunov function. This is the reason for using
the augmented energy
\[
    \mathcal L_k
    =
    F(v_k)-F^\star
    +b\norm{v_k-x^\star}^2
    +c\norm{x_k-x^\star}^2.
\]
\end{remark}

% \begin{remark}[Smooth case]
% When \(r=0\), the proximal operator becomes the identity map, and Prox-NAG-GS reduces to the
% corresponding smooth semi-implicit NAG-GS update in the constant-parameter regime considered
% in this section. Therefore, Theorem~\ref{thm:linear-conservative-new} also gives a linear
% convergence result for this smooth specialization, provided that \(\widehat\mu\ge L\), which remained unknown until now.
% \end{remark}
\begin{remark}[Smooth case]
When \(r=0\), the proximal operator becomes the identity map, and Prox-NAG-GS reduces to the
corresponding smooth semi-implicit NAG-GS update in the constant-parameter regime considered
in this section. Therefore, Theorem~\ref{thm:linear-conservative-new} also gives a linear
convergence result for this smooth specialization, provided that \(\widehat\mu\ge L\).

This should not be confused with an optimal accelerated rate, nor with a complete analysis of all
NAG-GS parameter regimes. It gives a conservative but rigorous linear-convergence regime for the
smooth semi-implicit update.
\end{remark}

\subsection{Convex case: best-iterate and averaged rates}
\label{subsec:convex-case}

We now record what the same analysis gives in the merely convex case. This result is not needed
for the linear convergence theorem above, but it is useful to complete the theoretical picture. It
shows that the augmented Lyapunov function is not specific to strong convexity: the additional
term \(\|x_k-x^\star\|^2\) is already the right quantity to obtain a descent estimate when
\(F=f+r\) is convex.

In this subsection, we assume that \(f\) is convex and \(L\)-smooth, that \(r\) is proper, closed,
convex and proximable, and that \(F=f+r\) admits a minimizer \(x^\star\). We keep the same
constant-parameter regime as in the previous subsections, with \(a\in(0,1)\), and we assume
\[
    \widehat\mu\ge L.
\]
Under this condition, the mismatch term between the gradient point \(x_{k+1}\) and the proximal
point \(v_{k+1}\) can be absorbed as before. We then obtain a descent inequality for the same
type of augmented Lyapunov function. As a consequence, the method satisfies an \(O(1/k)\) rate
for the best iterate and for the averaged iterate.

\begin{theorem}[Convex case]
\label{thm:convex-case}
Assume that Assumption~\ref{ass:deterministic-convex} holds. 
Assume also that Prox-NAG-GS is run in the constant-parameter regime
\[
    x_{k+1}=(1-a)x_k+av_k,
    \qquad
    z_{k+1}=(1-a)v_k+ax_{k+1},
\]
\[
    v_{k+1}
    =
    \prox_{\frac{a}{\widehat\mu}r}
    \left(
        z_{k+1}
        -
        \frac{a}{\widehat\mu}\nabla f(x_{k+1})
    \right),
\]
with \(a\in(0,1)\). Finally, assume that the algorithmic curvature parameter satisfies
\[
    \widehat\mu\ge L.
\]

Let
\[
    G_k:=F(v_k)-F^\star,
    \qquad
    V_k:=\norm{v_k-x^\star}^2,
    \qquad
    X_k:=\norm{x_k-x^\star}^2,
\]
and define
\[
    \mathcal E_k
    :=
    G_k
    +
    \frac{\widehat\mu}{2a}V_k
    +
    \frac{\widehat\mu(1-a)}{2a}X_k .
\]
Then, for every \(k\ge0\),
\begin{equation}
\label{eq:convex-lyap-descent}
    \mathcal E_{k+1}
    \le
    \mathcal E_k
    -
    G_k
    -
    \frac{\widehat\mu(1-a)}{2}\norm{x_k-v_k}^2 .
\end{equation}
Consequently,
\begin{equation}
\label{eq:convex-sum-gap}
    \sum_{i=0}^{k-1}
    \left(F(v_i)-F^\star\right)
    \le
    \mathcal E_0,
    \qquad k\ge1,
\end{equation}
and
\begin{equation}
\label{eq:convex-sum-distance}
    \sum_{i=0}^{k-1}
    \norm{x_i-v_i}^2
    \le
    \frac{2\mathcal E_0}{\widehat\mu(1-a)} .
\end{equation}
In particular,
\begin{equation}
\label{eq:convex-best-rate}
    \min_{0\le i\le k-1}
    \left(F(v_i)-F^\star\right)
    \le
    \frac{\mathcal E_0}{k}.
\end{equation}
Moreover, if
\[
    \bar v_k:=\frac1k\sum_{i=0}^{k-1}v_i,
\]
then
\begin{equation}
\label{eq:convex-ergodic-rate}
    F(\bar v_k)-F^\star
    \le
    \frac{\mathcal E_0}{k}.
\end{equation}
Finally,
\[
    F(v_k)\to F^\star .
\]
\end{theorem}

\begin{proof}
Since \(f\) is only assumed convex in this subsection, we use the reduced one-step inequality
with \(\mu_f=0\). From Proposition~\ref{prop:reduced-one-step-new}, and since
\(\widehat\mu\ge L\), we have
\begin{equation}
\label{eq:convex-reduced-proof}
    G_{k+1}
    +
    \frac{\widehat\mu}{2a}V_{k+1}
    \le
    \frac{\widehat\mu(1-a)}{2a}V_k
    +
    \frac{\widehat\mu}{2}X_{k+1}.
\end{equation}
Set
\[
    b:=\frac{\widehat\mu}{2a},
    \qquad
    c:=\frac{\widehat\mu(1-a)}{2a}=b(1-a).
\]
Then \(\widehat\mu/2=ab\), and \eqref{eq:convex-reduced-proof} becomes
\[
    G_{k+1}+bV_{k+1}
    \le
    b(1-a)V_k+abX_{k+1}.
\]
Adding \(cX_{k+1}\) to both sides gives
\[
    \mathcal E_{k+1}
    \le
    b(1-a)V_k+(ab+c)X_{k+1}.
\]
Since \(c=b(1-a)\), we have
\[
    ab+c=ab+b(1-a)=b.
\]
Therefore,
\begin{equation}
\label{eq:convex-before-x-identity}
    \mathcal E_{k+1}
    \le
    b(1-a)V_k+bX_{k+1}.
\end{equation}

We now use the identity
\[
    X_{k+1}
    =
    (1-a)X_k+aV_k-a(1-a)\norm{x_k-v_k}^2 .
\]
Substituting this identity into \eqref{eq:convex-before-x-identity}, we obtain
\[
\begin{aligned}
    \mathcal E_{k+1}
    &\le
    b(1-a)V_k
    +
    b\left((1-a)X_k+aV_k-a(1-a)\norm{x_k-v_k}^2\right)\\
    &=
    bV_k
    +
    b(1-a)X_k
    -
    ab(1-a)\norm{x_k-v_k}^2 .
\end{aligned}
\]
Using \(c=b(1-a)\) and \(ab=\widehat\mu/2\), this becomes
\[
    \mathcal E_{k+1}
    \le
    bV_k+cX_k
    -
    \frac{\widehat\mu(1-a)}{2}\norm{x_k-v_k}^2 .
\]
Since
\[
    \mathcal E_k=G_k+bV_k+cX_k,
\]
we obtain
\[
    \mathcal E_{k+1}
    \le
    \mathcal E_k
    -
    G_k
    -
    \frac{\widehat\mu(1-a)}{2}\norm{x_k-v_k}^2,
\]
which proves \eqref{eq:convex-lyap-descent}.

Summing \eqref{eq:convex-lyap-descent} from \(i=0\) to \(k-1\) gives
\[
    \mathcal E_k
    +
    \sum_{i=0}^{k-1}G_i
    +
    \frac{\widehat\mu(1-a)}{2}
    \sum_{i=0}^{k-1}\norm{x_i-v_i}^2
    \le
    \mathcal E_0.
\]
Since \(\mathcal E_k\ge0\), this implies \eqref{eq:convex-sum-gap} and
\eqref{eq:convex-sum-distance}. The best-iterate estimate follows immediately:
\[
    k\min_{0\le i\le k-1}G_i
    \le
    \sum_{i=0}^{k-1}G_i
    \le
    \mathcal E_0.
\]

For the averaged iterate, convexity of \(F\) gives
\[
    F(\bar v_k)
    =
    F\left(\frac1k\sum_{i=0}^{k-1}v_i\right)
    \le
    \frac1k\sum_{i=0}^{k-1}F(v_i).
\]
Thus
\[
    F(\bar v_k)-F^\star
    \le
    \frac1k\sum_{i=0}^{k-1}\left(F(v_i)-F^\star\right)
    \le
    \frac{\mathcal E_0}{k}.
\]
Finally, since the nonnegative series \(\sum_{i=0}^{\infty}G_i\) is bounded by
\(\mathcal E_0\), we have \(G_k\to0\), that is, \(F(v_k)\to F^\star\).
\end{proof}

\begin{remark}
The theorem gives an \(O(1/k)\) rate for the best iterate and for the averaged iterate. It does
not claim an \(O(1/k)\) rate for the last iterate \(v_k\). The descent inequality implies
\(F(v_k)\to F^\star\), but a last-iterate rate would require an additional argument, for instance
a monotonicity property or a sharper control of the objective gaps.
\end{remark}

\section{Numerical tests}
\label{sec:numerics}

We test Prox-NAG-GS on four composite optimization benchmarks and one additional diagnostic experiment. The main benchmarks are intended to evaluate the practical behavior of the method. The deterministic tests use Elastic Net and Group Lasso objectives. The stochastic tests use softmax regression on MNIST, with either an entrywise \(\ell_1\) penalty or a Group Lasso penalty. These tests are intended to check two points: whether the proximal extension keeps the fast behavior observed for NAG-GS in deterministic problems, and how it behaves in stochastic learning when the nonsmooth term promotes sparsity.

In addition to these performance-oriented tests, we include a separate deterministic check of the theoretical regime. In the convergence proof, the algorithmic curvature parameter is required to satisfy the sufficient condition \(\widehat\mu\ge L\). In the main numerical comparisons, however, \(\widehat\mu\) is tuned empirically. The purpose of the additional test is therefore not to improve the best performance, but to verify that the conservative regime covered by the theory is numerically meaningful. We run Prox-NAG-GS with \(\widehat\mu=L\), monitor the objective gaps at both \(x_k\) and \(v_k\), and record the augmented Lyapunov quantity used in Theorem~\ref{thm:linear-conservative-new}.

All results are averaged over five random seeds. In the main performance benchmarks, all methods are tuned with Optuna~\cite{Akiba2019Optuna}, using the same tuning budget inside each benchmark. In the deterministic tests, the tuning criterion is based on the optimality gap. In the stochastic tests, the tuning criterion is based on validation performance. After each epoch, we evaluate the full regularized objective on the whole training set. Moreover, in all plots, solid curves show averages over five seeds, and shaded regions show one standard deviation.

For deterministic problems, we report the optimality gap \(F(x_k)-F^\star\), where \(F^\star\) is computed by a high-precision reference solver. Note that for Prox-NAG-GS, the deterministic performance plots report the objective value at \(x_k\), which is the practical output iterate in our implementation. The theory is stated for \(v_k\), the proximal variable used in the Lyapunov analysis. At a fixed point of the deterministic iteration, the two variables coincide. In the additional theory-regime test, we report both \(F(x_k)-F^\star\) and \(F(v_k)-F^\star\), precisely to connect the numerical output with the variable used in the proof. For stochastic problems, we report the full objective, the data-fit term, the regularization term, test accuracy, sparsity, and wall-clock time. All experiments were run on Google Colab using an A100 GPU. The code used to generate all tables and figures is publicly available; see the reproducibility paragraph at the end of this section.

\subsection{Deterministic Elastic Net}

The first deterministic benchmark is the Elastic Net problem
\begin{equation}
\label{eq:elastic-net-num}
    \min_{x\in\R^d}
    \frac12\norm{Ax-b}_2^2
    +\frac{\lambda_2}{2}\norm{x}_2^2
    +\lambda_1\norm{x}_1 .
\end{equation}
We compare Prox-NAG-GS with ISTA~\cite{DaubechiesDefriseDeMol2004},
FISTA~\cite{BeckTeboulle2009}, and Chambolle-Pock~\cite{ChambollePock2011}. We use two instances. In the first one, \(A\) is generated as a Gaussian random matrix. 
In the second one, \(A\) is a synthetic random matrix with controlled condition number close to \(10^3\). 
We refer to the two instances as the easy and hard cases, respectively.

\begin{table}[!htbp]
\centering
\caption{Deterministic Elastic Net. Final objective, iterations to reach the \(10^{-6}\) optimality gap, and wall-clock time. Values are averages over five seeds.}
\label{tab:elastic-net-summary}
\begin{tabular}{llccc}
\toprule
Instance & Method & Final obj. & Iterations to \(10^{-6}\) & Time (s) \\
\midrule
Easy & ISTA & 4.1050 & 95.4 & 0.1183 \\
Easy & FISTA & 4.1050 & 68.4 & 0.1163 \\
Easy & Chambolle-Pock & 4.1050 & 67.4 & \textbf{0.1016} \\
Easy & Prox-NAG-GS & \textbf{4.1050} & \textbf{28.0} & 0.1220 \\
\midrule
Hard & ISTA & 0.6798 & 100.0 & 0.1976 \\
Hard & FISTA & 0.6798 & 61.6 & \textbf{0.1897} \\
Hard & Chambolle-Pock & 0.6798 & 93.6 & 0.2063 \\
Hard & Prox-NAG-GS & \textbf{0.6798} & \textbf{24.4} & 0.2131 \\
\bottomrule
\end{tabular}
\end{table}

\FloatBarrier

Table~\ref{tab:elastic-net-summary} shows that Prox-NAG-GS reaches the same final objective as the baselines while requiring substantially fewer iterations. At the \(10^{-6}\) gap threshold, the iteration reduction ranges from \(2.41\times\) to \(3.41\times\) on the easy instance and from \(2.52\times\) to \(4.10\times\) on the hard instance.

\begin{figure}[!htbp]
\centering
\includegraphics[width=0.49\textwidth]{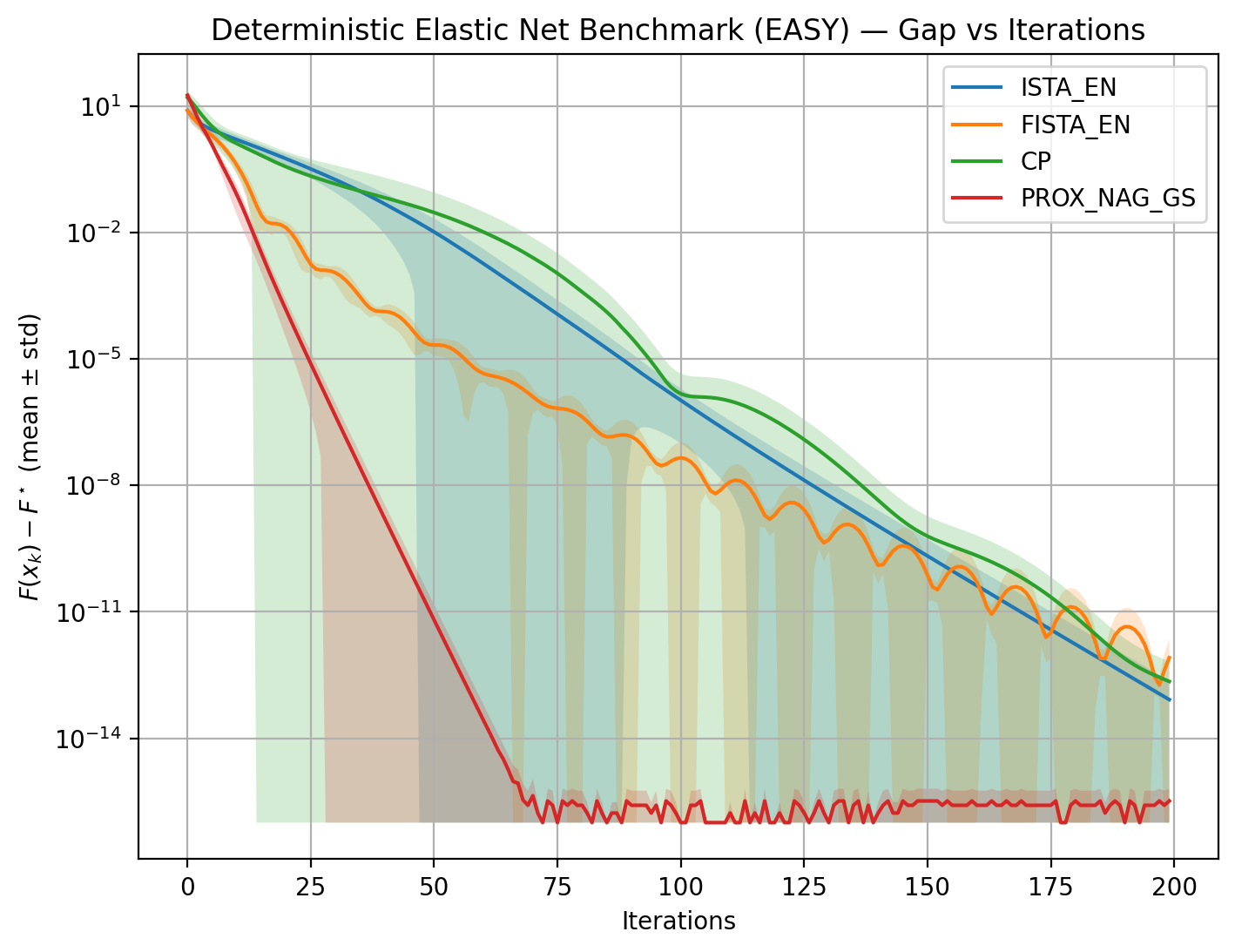}
\includegraphics[width=0.49\textwidth]{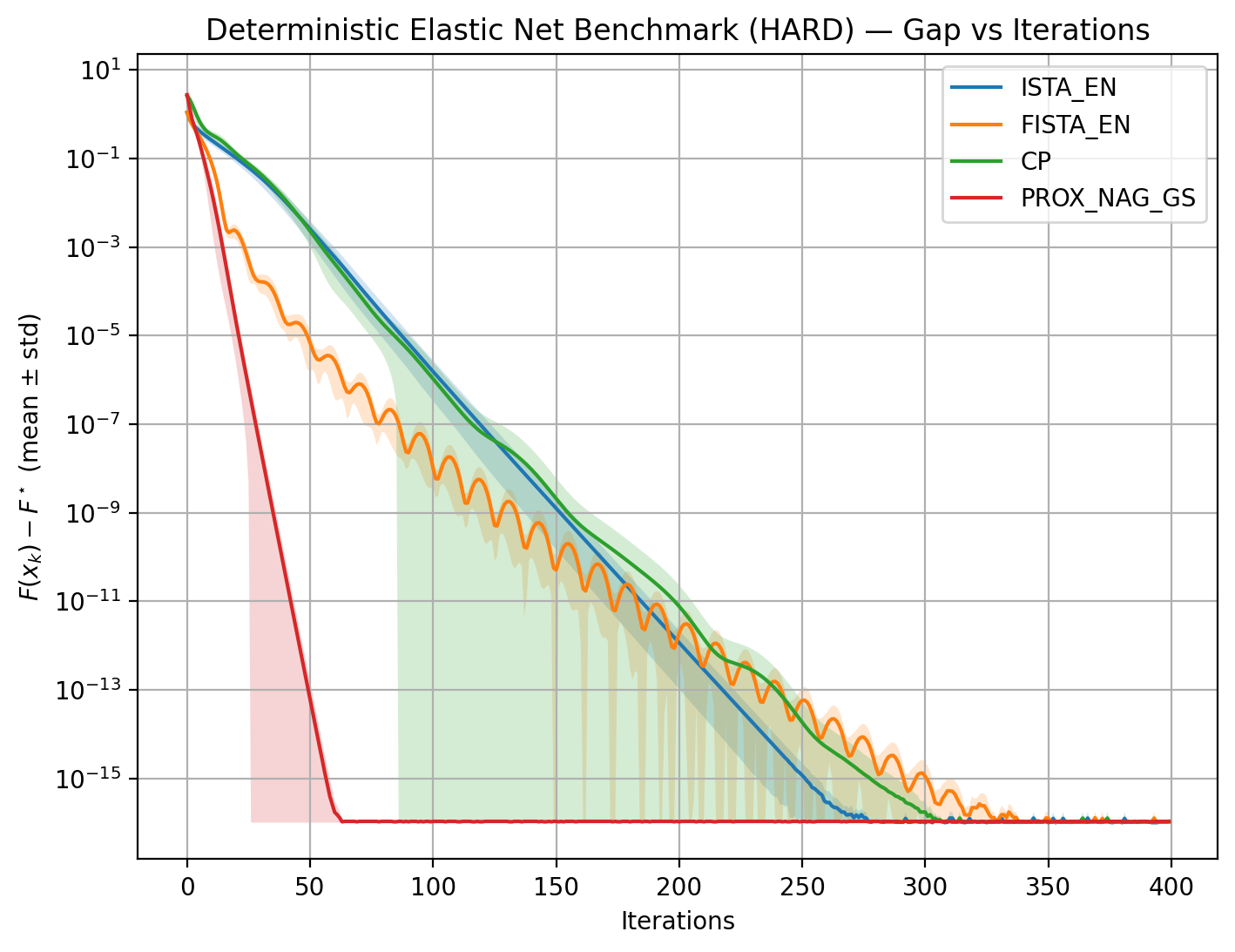}
\caption{Deterministic Elastic Net. Optimality gap \(F(x_k)-F^\star\) versus iterations. Left: easy instance. Right: hard instance.}
\label{fig:elastic-net-gap-iter}
\end{figure}

\begin{figure}[!htbp]
\centering
\includegraphics[width=0.49\textwidth]{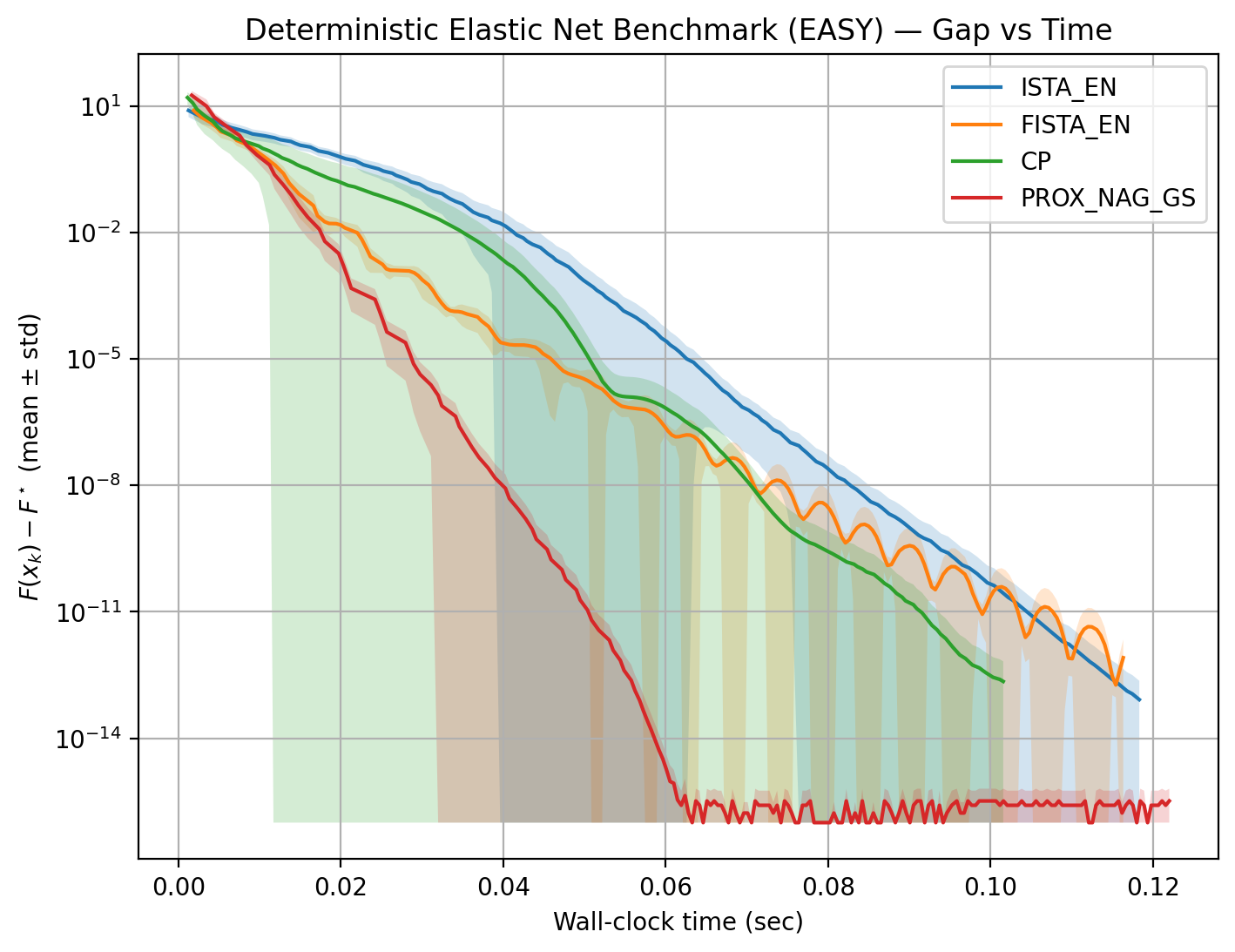}
\includegraphics[width=0.49\textwidth]{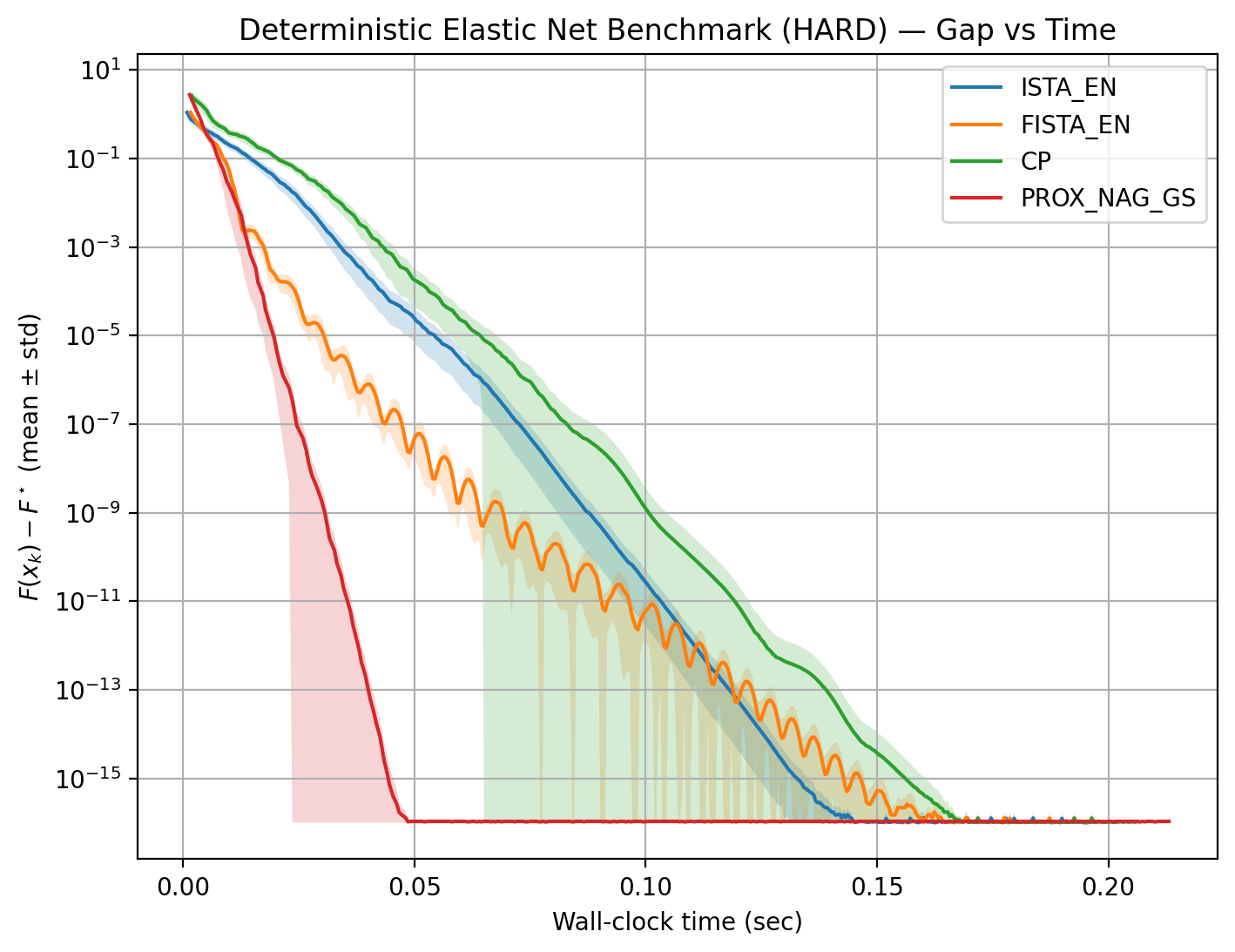}
\caption{Deterministic Elastic Net. Optimality gap \(F(x_k)-F^\star\) versus wall-clock time. Left: easy instance. Right: hard instance.}
\label{fig:elastic-net-gap-time}
\end{figure}

\FloatBarrier

All methods reach the same final objective up to the reported precision. Prox-NAG-GS reaches
the \(10^{-6}\) gap threshold with substantially fewer iterations. On the easy instance, the
iteration reduction ranges from \(2.41\times\) to \(3.41\times\), depending on the baseline. On the
hard instance, where the synthetic matrix has condition number close to \(10^3\), it ranges from
\(2.52\times\) to \(4.10\times\).

The wall-clock comparison is less favorable in this benchmark. Although Prox-NAG-GS needs
fewer iterations, one iteration is slightly more expensive in the present implementation. Hence,
the Elastic Net experiment mainly shows an iteration-efficiency advantage, rather than a clear
runtime advantage.

\FloatBarrier

\subsection{Deterministic Group Lasso}

We next consider the Group Lasso problem
\begin{equation}
\label{eq:group-lasso-det-num}
    \min_{x\in\R^d}
    \frac12\norm{Ax-b}_2^2
    +\frac{\lambda_2}{2}\norm{x}_2^2
    +\lambda_g\sum_{G\in\mathcal G}\norm{x_G}_2 .
\end{equation}
The groups are contiguous blocks of size \(10\). The ground truth contains \(8\) active groups.
This benchmark tests whether the proximal NAG-GS update remains effective for a structured
nonsmooth penalty.

% \begin{table}[t]
% \centering
% \caption{Deterministic Group Lasso. Final objective, active groups, iterations to reach the \(10^{-6}\) optimality gap, and wall-clock time. Values are averages over five seeds.}
% \label{tab:group-lasso-det-summary}
% \begin{tabular}{llcccc}
% \toprule
% Instance & Method & Final obj. & Active groups & Iterations to \(10^{-6}\) & Time (s) \\
% \midrule
% Easy & Group-ISTA & 5.6159 & 8.0 & 125.8 & 0.2302 \\
% Easy & Group-FISTA & 5.6159 & 8.0 & 93.2 & 0.2250 \\
% Easy & Group-CP & 5.6159 & 8.0 & 69.0 & 0.2304 \\
% Easy & Group Prox-NAG-GS & \textbf{5.6159} & 8.0 & \textbf{34.0} & \textbf{0.1846} \\
% \midrule
% Hard & Group-ISTA & 0.9366 & 9.0 & 161.6 & 0.4215 \\
% Hard & Group-FISTA & 0.9366 & 9.0 & 88.6 & 0.4225 \\
% Hard & Group-CP & 0.9366 & 9.0 & 140.4 & 0.4220 \\
% Hard & Group Prox-NAG-GS & \textbf{0.9366} & 9.0 & \textbf{30.8} & \textbf{0.3768} \\
% \bottomrule
% \end{tabular}
% \end{table}
\begin{table}[H]
\centering
\caption{Deterministic Group Lasso. Final objective, active groups, iterations to reach the \(10^{-6}\) optimality gap, and wall-clock time. Values are averages over five seeds.}
\label{tab:group-lasso-det-summary}
\small
\setlength{\tabcolsep}{4pt}
\begin{tabular}{llcccc}
\toprule
Instance & Method & Final obj. & Active groups & Iter. to \(10^{-6}\) & Time (s) \\
\midrule
Easy & Group-ISTA & 5.6159 & 8.0 & 125.8 & 0.2302 \\
Easy & Group-FISTA & 5.6159 & 8.0 & 93.2 & 0.2250 \\
Easy & Group-CP & 5.6159 & 8.0 & 69.0 & 0.2304 \\
Easy & Group Prox-NAG-GS & \textbf{5.6159} & 8.0 & \textbf{34.0} & \textbf{0.1846} \\
\midrule
Hard & Group-ISTA & 0.9366 & 9.0 & 161.6 & 0.4215 \\
Hard & Group-FISTA & 0.9366 & 9.0 & 88.6 & 0.4225 \\
Hard & Group-CP & 0.9366 & 9.0 & 140.4 & 0.4220 \\
Hard & Group Prox-NAG-GS & \textbf{0.9366} & 9.0 & \textbf{30.8} & \textbf{0.3768} \\
\bottomrule
\end{tabular}
\end{table}

\FloatBarrier

\begin{figure}[!htbp]
\centering
\includegraphics[width=0.49\textwidth]{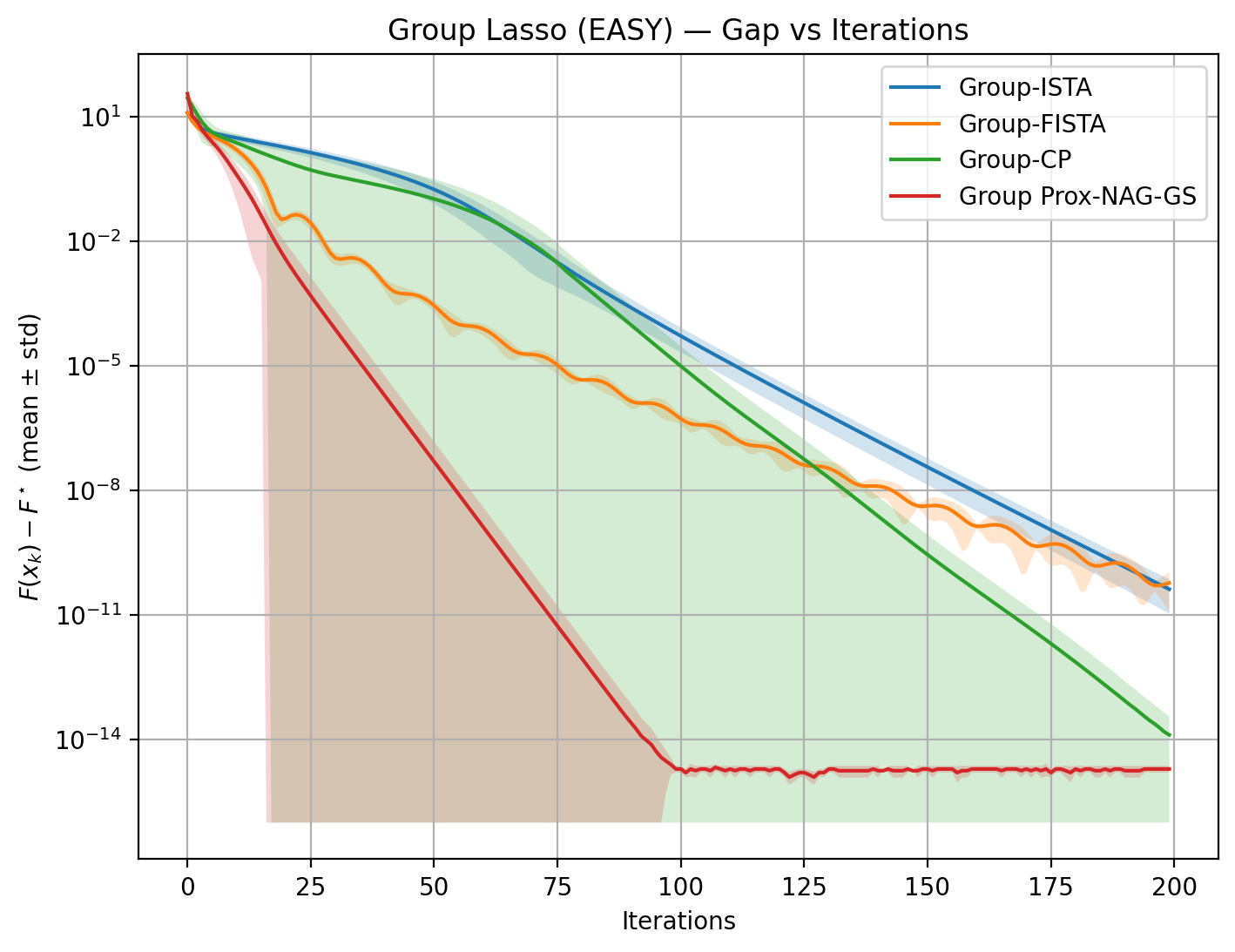}
\includegraphics[width=0.49\textwidth]{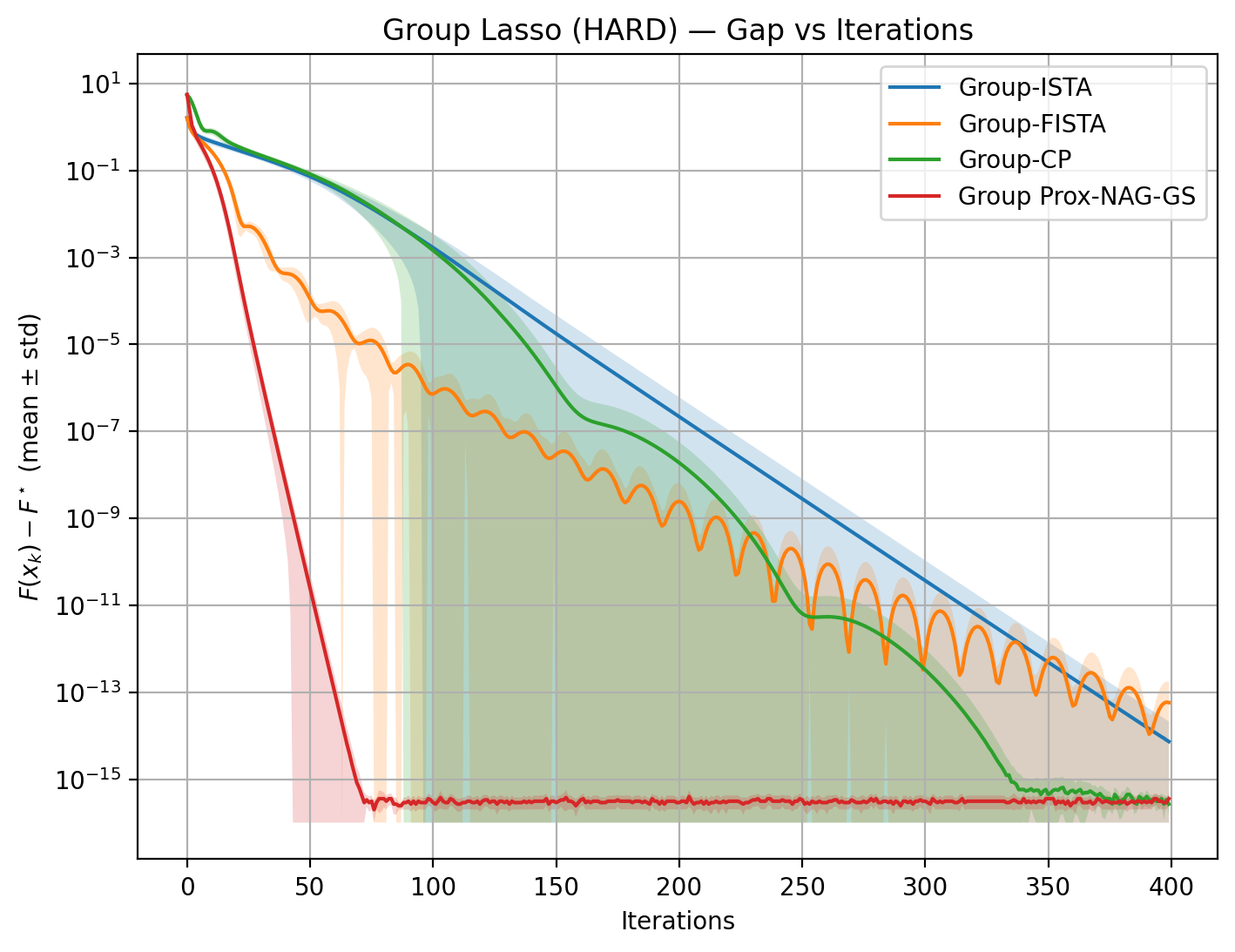}
\caption{Deterministic Group Lasso. Optimality gap \(F(x_k)-F^\star\) versus iterations. Left: easy instance. Right: hard instance.}
\label{fig:group-lasso-gap-iter}
\end{figure}

\begin{figure}[!htbp]
\centering
\includegraphics[width=0.49\textwidth]{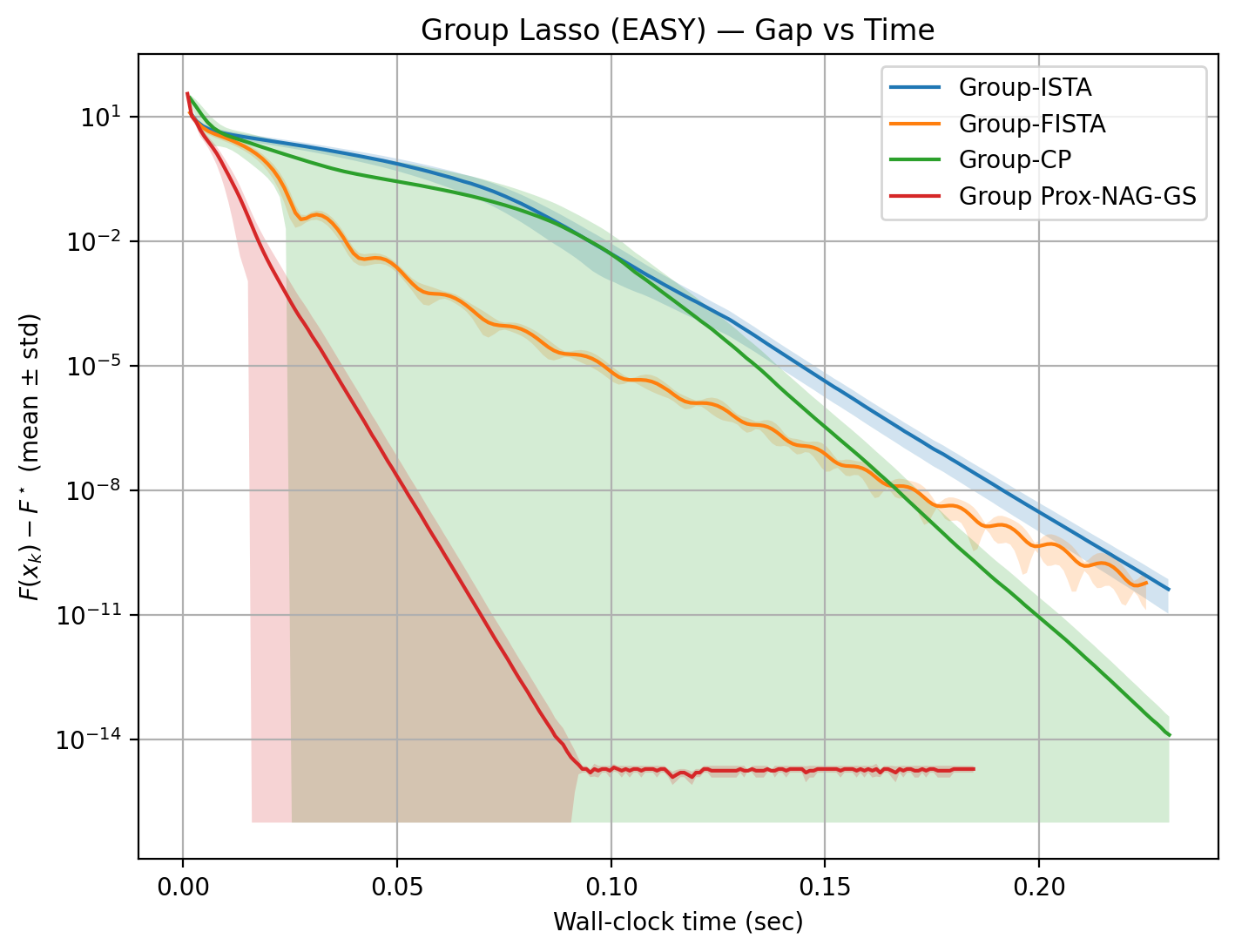}
\includegraphics[width=0.49\textwidth]{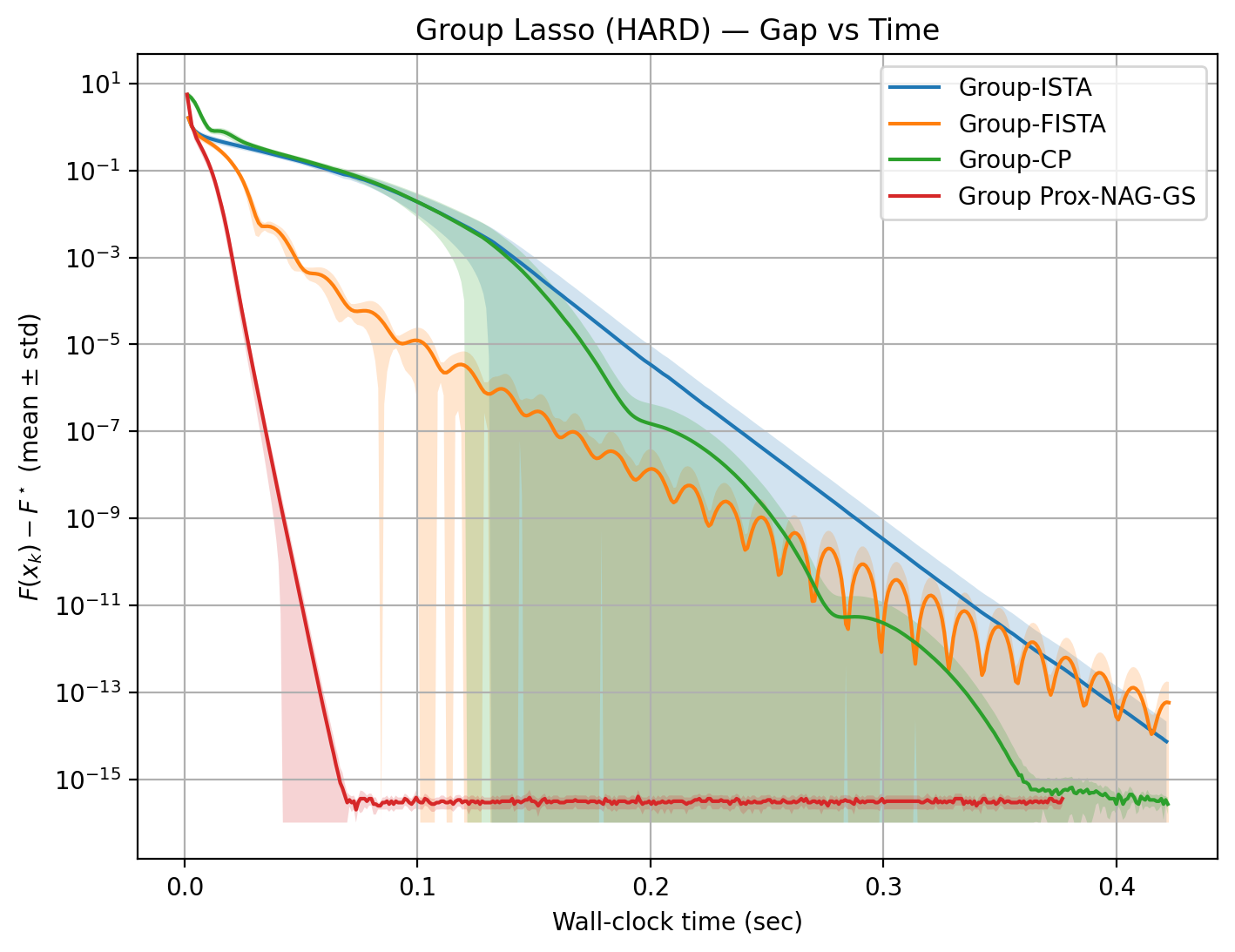}
\caption{Deterministic Group Lasso. Optimality gap \(F(x_k)-F^\star\) versus wall-clock time. Left: easy instance. Right: hard instance.}
\label{fig:group-lasso-gap-time}
\end{figure}

\FloatBarrier

% This benchmark gives the clearest deterministic evidence in favor of Prox-NAG-GS. All methods
% reach the same objective value and recover the same number of active groups. However,
% Prox-NAG-GS reaches the \(10^{-6}\) gap threshold in significantly fewer iterations. In contrast
% with the Elastic Net test, this also translates into the best wall-clock time on both the easy and
% the ill-conditioned instances.
This benchmark gives the clearest deterministic evidence in favor of Prox-NAG-GS. All methods
reach the same objective value and recover the same number of active groups. However,
Prox-NAG-GS reaches the \(10^{-6}\) gap threshold in significantly fewer iterations. In contrast
with the Elastic Net test, this also translates into the best wall-clock time on both the easy and
the hard instances.

\FloatBarrier

\subsection{Numerical check of the theoretical regime}
\label{subsec:theory-regime-check}

The previous deterministic tests use tuned parameters and are meant to compare Prox-NAG-GS with standard proximal baselines. We now perform a different experiment. Its goal is to check the parameter regime used in the convergence proof. More precisely, we run Prox-NAG-GS on the deterministic Elastic Net problem with
\[
    \widehat\mu=L,
\]
where \(L\) is the smoothness constant of the differentiable part. This is the boundary case of the sufficient condition \(\widehat\mu\ge L\) used in Theorem~\ref{thm:linear-conservative-new}.

This experiment is not intended as a performance comparison. The tuned values of \(\widehat\mu\) used in the main benchmarks can be more aggressive than the conservative value covered by the theorem. The purpose here is instead to verify that the theoretical regime gives the behavior predicted by the analysis.

Figure~\ref{fig:theory-regime-check} reports the results. The left panel shows the objective gaps at both variables. The gap at \(v_k\) is the one directly connected to the Lyapunov analysis, while \(x_k\) is the practical output variable used in the deterministic experiments. Both gaps decrease steadily, and the two curves become almost indistinguishable after a short initial phase. This is consistent with the fact that \(x_k\) and \(v_k\) coincide at a fixed point of the deterministic iteration.

The right panel reports the augmented Lyapunov quantity
\[
    \mathcal L_k
    =
    F(v_k)-F^\star
    +b\norm{v_k-x^\star}^2
    +c\norm{x_k-x^\star}^2 .
\]
It decreases along the iterations and remains below the theoretical envelope \(\mathcal L_0\theta^k\). For each seed, we also checked the contraction inequality
\[
    \mathcal L_{k+1}\le \theta \mathcal L_k,
\]
where \(\theta\) is the value given by Theorem~\ref{thm:linear-conservative-new}. No violation was observed in the five runs. Finally, the mismatch-absorption term from Lemma~\ref{lem:mismatch-absorption-new} remained nonpositive up to numerical precision for all iterations and all seeds. This gives a direct numerical check that the conservative condition \(\widehat\mu\ge L\) leads to the behavior predicted by the proof.

\begin{figure}[!htbp]
\centering
\includegraphics[width=0.49\textwidth]{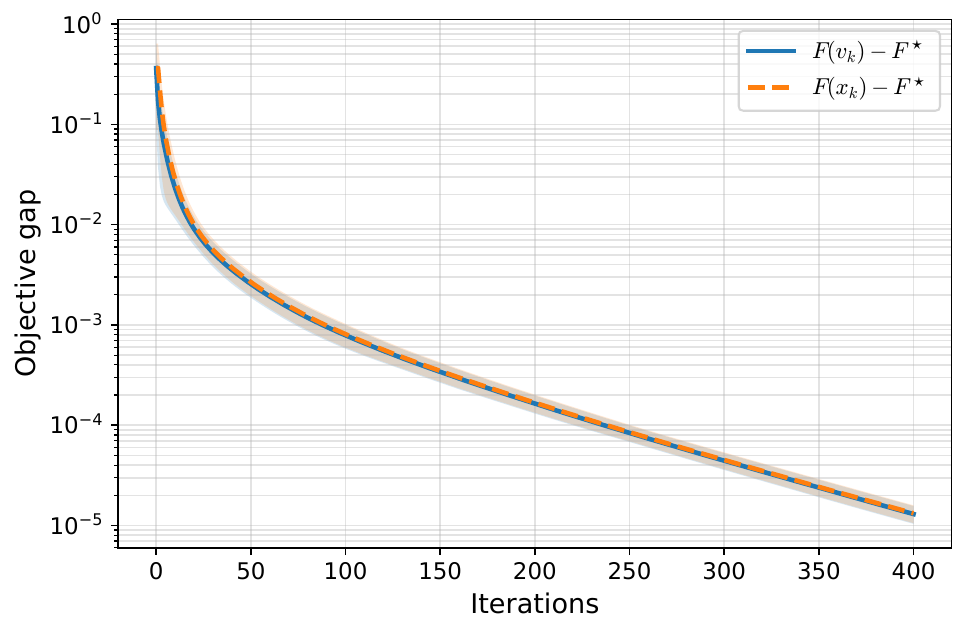}
\includegraphics[width=0.49\textwidth]{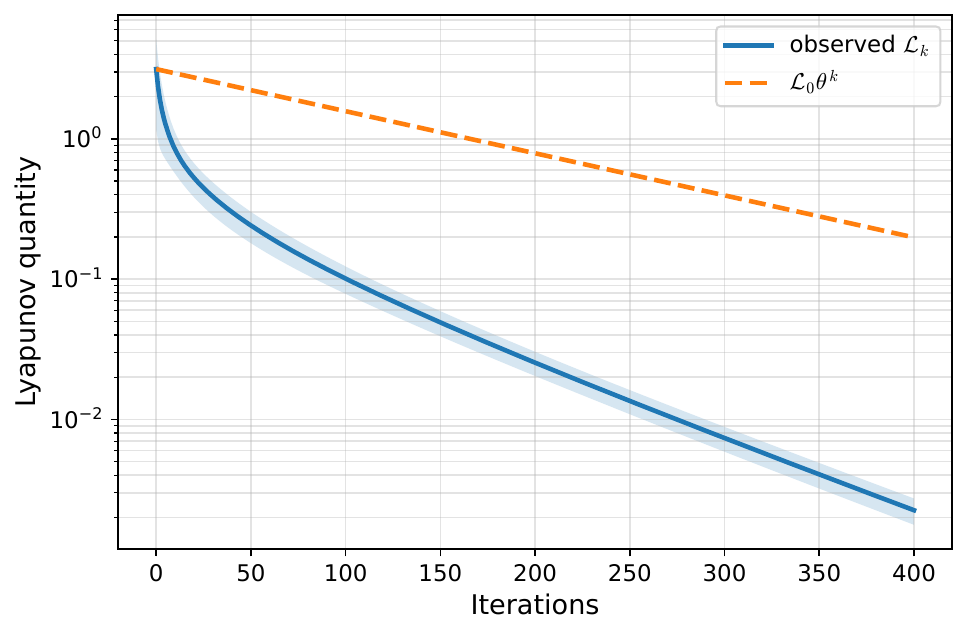}
\caption{Numerical check of the theoretical regime on the deterministic Elastic Net benchmark with \(\widehat\mu=L\). Left: objective gaps at \(x_k\) and \(v_k\). Right: augmented Lyapunov quantity \(\mathcal L_k\) compared with the theoretical envelope \(\mathcal L_0\theta^k\). Curves are averaged over five seeds, and shaded regions show one standard deviation.}
\label{fig:theory-regime-check}
\end{figure}

\FloatBarrier

\subsection{Stochastic \(\ell_1\) softmax regression}

The stochastic \(\ell_1\) benchmark uses softmax regression on MNIST:
\begin{equation}
\label{eq:stoch-l1-num}
    \min_{W\in\R^{d\times C}}
    \frac1n\sum_{i=1}^n \ell_i(W)
    +\lambda_1\norm{W}_1
    +\frac{\lambda_2}{2}\norm{W}_F^2 .
\end{equation}
We compare Prox-NAG-GS with Prox-SGD. Both methods use the same mini-batch size and the
same train-validation-test split for each seed.

\begin{table}[H]
\centering
\caption{Stochastic \(\ell_1\) softmax regression. Mean \(\pm\) standard deviation over five seeds.}
\label{tab:stoch-l1-summary}
\resizebox{\textwidth}{!}{
\begin{tabular}{lcccccc}
\toprule
Method & Final obj. & Data-fit & Reg. term & Test acc. & Sparsity & Time (s) \\
\midrule
Prox-SGD
& \(0.3362\pm0.0008\)
& \(0.2802\pm0.0007\)
& \(\mathbf{0.0561\pm0.0002}\)
& \(0.9206\pm0.0008\)
& \(\mathbf{0.4127\pm0.0035}\)
& \(\mathbf{37.53\pm0.68}\) \\
Prox-NAG-GS
& \(\mathbf{0.3351\pm0.0049}\)
& \(\mathbf{0.2723\pm0.0025}\)
& \(0.0628\pm0.0027\)
& \(\mathbf{0.9212\pm0.0017}\)
& \(0.3402\pm0.0247\)
& \(37.92\pm0.32\) \\
\bottomrule
\end{tabular}
}
\end{table}

\FloatBarrier

% \begin{table}[t]
% \centering
% \caption{\(\lambda_1\) sweep for stochastic \(\ell_1\) softmax regression. Values are averages over five seeds.}
% \label{tab:lambda-sweep}
% \begin{tabular}{llccc}
% \toprule
% \(\lambda_1\) & Method & Final test acc. & Test acc. at best val. & Final sparsity \\
% \midrule
% \(10^{-4}\) & Prox-SGD & 0.92056 & 0.92090 & 0.41268 \\
% \(10^{-4}\) & Prox-NAG-GS & 0.92124 & 0.92144 & 0.34015 \\
% \midrule
% \(5\times10^{-4}\) & Prox-SGD & 0.90966 & 0.91020 & 0.54074 \\
% \(5\times10^{-4}\) & Prox-NAG-GS & 0.90854 & 0.91110 & 0.38997 \\
% \midrule
% \(10^{-3}\) & Prox-SGD & 0.89938 & 0.90070 & 0.62401 \\
% \(10^{-3}\) & Prox-NAG-GS & 0.89994 & 0.90558 & 0.43020 \\
% \midrule
% \(5\times10^{-3}\) & Prox-SGD & 0.85334 & 0.85752 & 0.83151 \\
% \(5\times10^{-3}\) & Prox-NAG-GS & 0.58604 & 0.88084 & 0.45298 \\
% \bottomrule
% \end{tabular}
% \end{table}
\begin{table}[H]
\centering
\caption{\(\lambda_1\) sweep for stochastic \(\ell_1\) softmax regression. Values are averages over five seeds.}
\label{tab:lambda-sweep}
\small
\setlength{\tabcolsep}{4pt}
\begin{tabular}{@{}llccc@{}}
\toprule
\(\lambda_1\) & Method & Final test acc. & Best-val test acc. & Sparsity \\
\midrule
\(10^{-4}\) & Prox-SGD & 0.92056 & 0.92090 & 0.41268 \\
\(10^{-4}\) & Prox-NAG-GS & 0.92124 & 0.92144 & 0.34015 \\
\midrule
\(5\times10^{-4}\) & Prox-SGD & 0.90966 & 0.91020 & 0.54074 \\
\(5\times10^{-4}\) & Prox-NAG-GS & 0.90854 & 0.91110 & 0.38997 \\
\midrule
\(10^{-3}\) & Prox-SGD & 0.89938 & 0.90070 & 0.62401 \\
\(10^{-3}\) & Prox-NAG-GS & 0.89994 & 0.90558 & 0.43020 \\
\midrule
\(5\times10^{-3}\) & Prox-SGD & 0.85334 & 0.85752 & 0.83151 \\
\(5\times10^{-3}\) & Prox-NAG-GS & 0.58604 & 0.88084 & 0.45298 \\
\bottomrule
\end{tabular}
\end{table}

\begin{figure}[!htbp]
\centering
\includegraphics[width=0.62\textwidth]{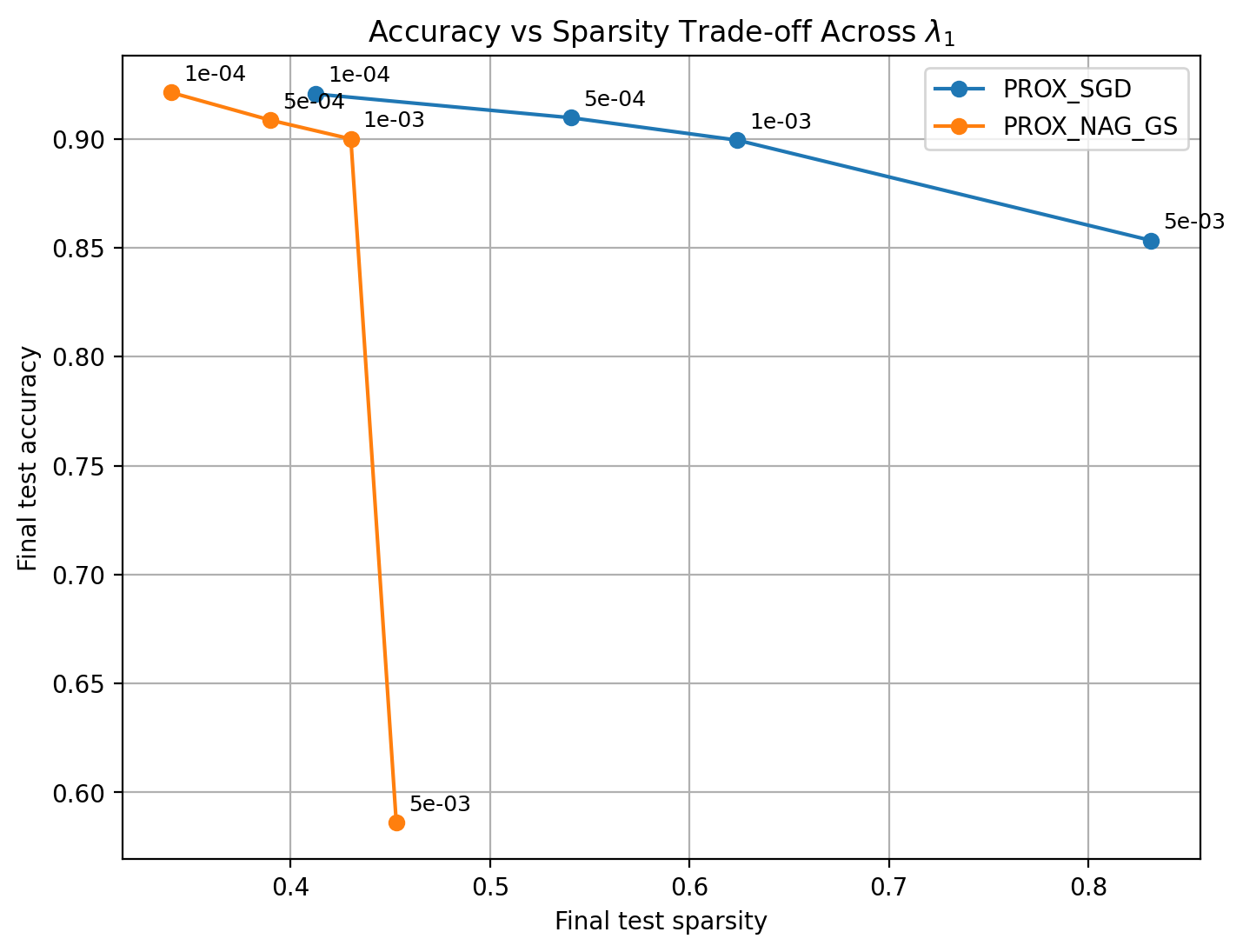}
\caption{Stochastic \(\ell_1\) softmax regression. Accuracy-sparsity trade-off across different values of \(\lambda_1\).}
\label{fig:accuracy-sparsity-tradeoff}
\end{figure}

\FloatBarrier

\begin{figure}[!htbp]
\centering
\includegraphics[width=0.48\textwidth]{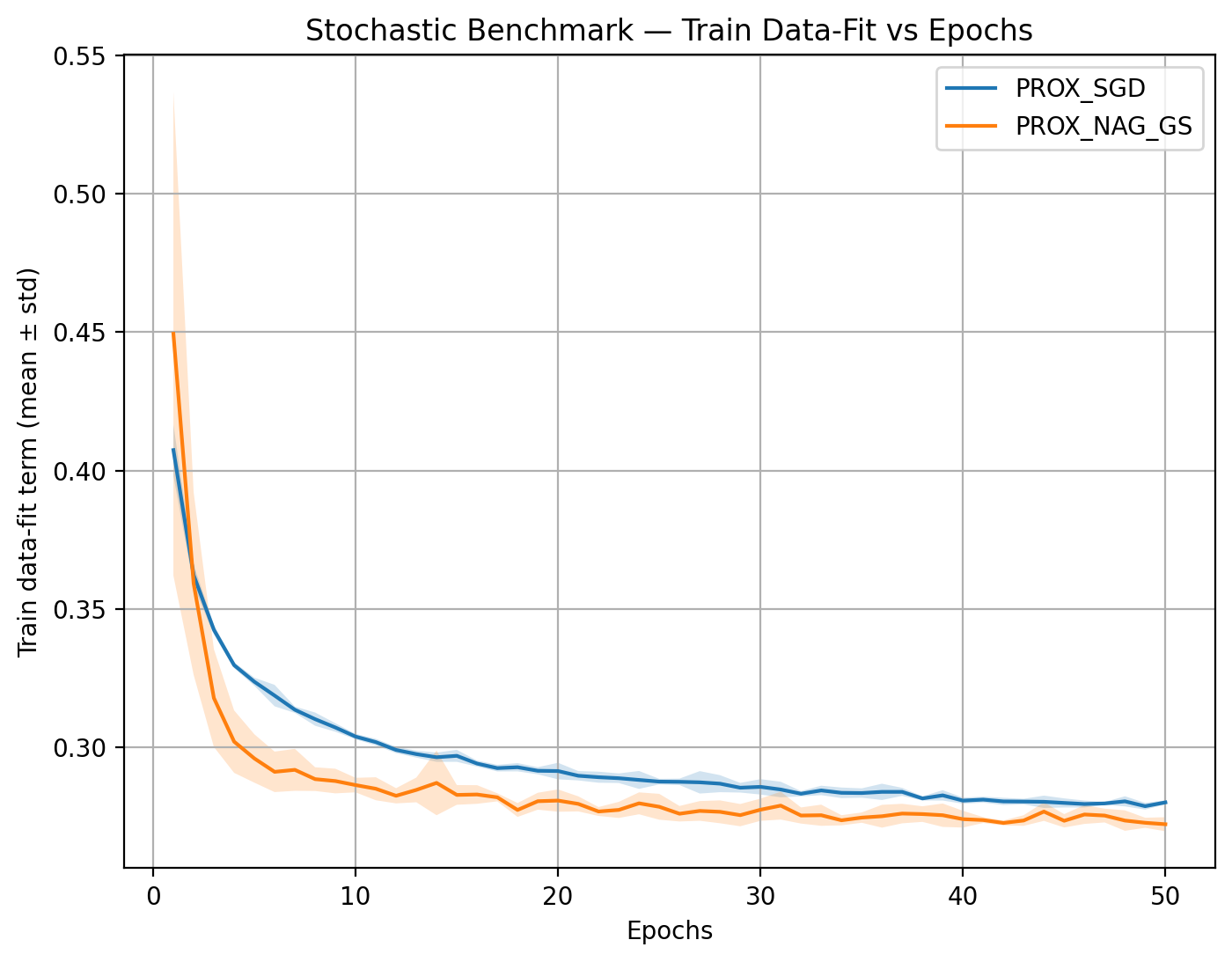}
\includegraphics[width=0.48\textwidth]{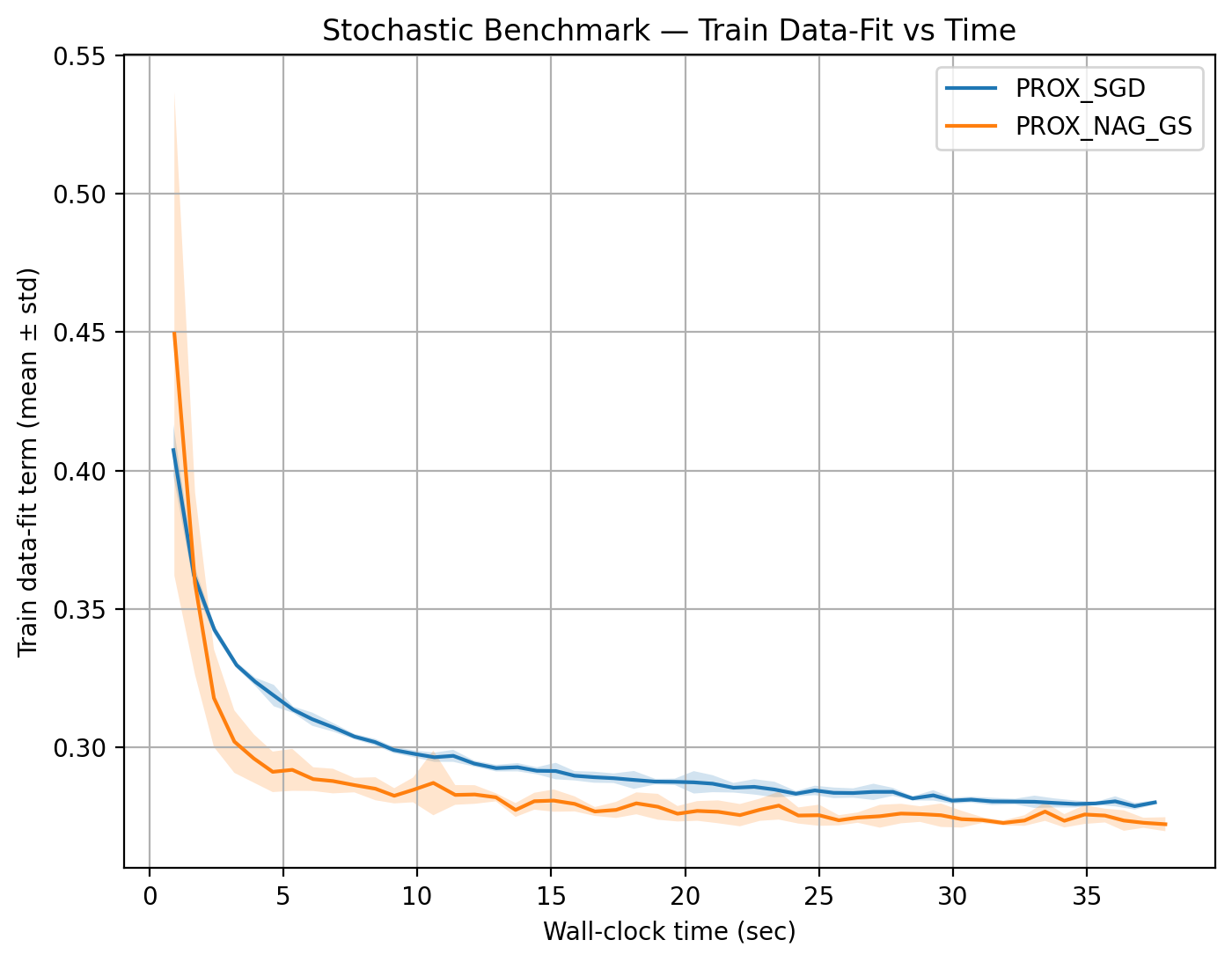}
\caption{Stochastic \(\ell_1\) softmax regression. Training data-fit term versus epochs and wall-clock time.}
\label{fig:stoch-l1-datafit}
\end{figure}

\FloatBarrier

\begin{figure}[!htbp]
\centering
\includegraphics[width=0.48\textwidth]{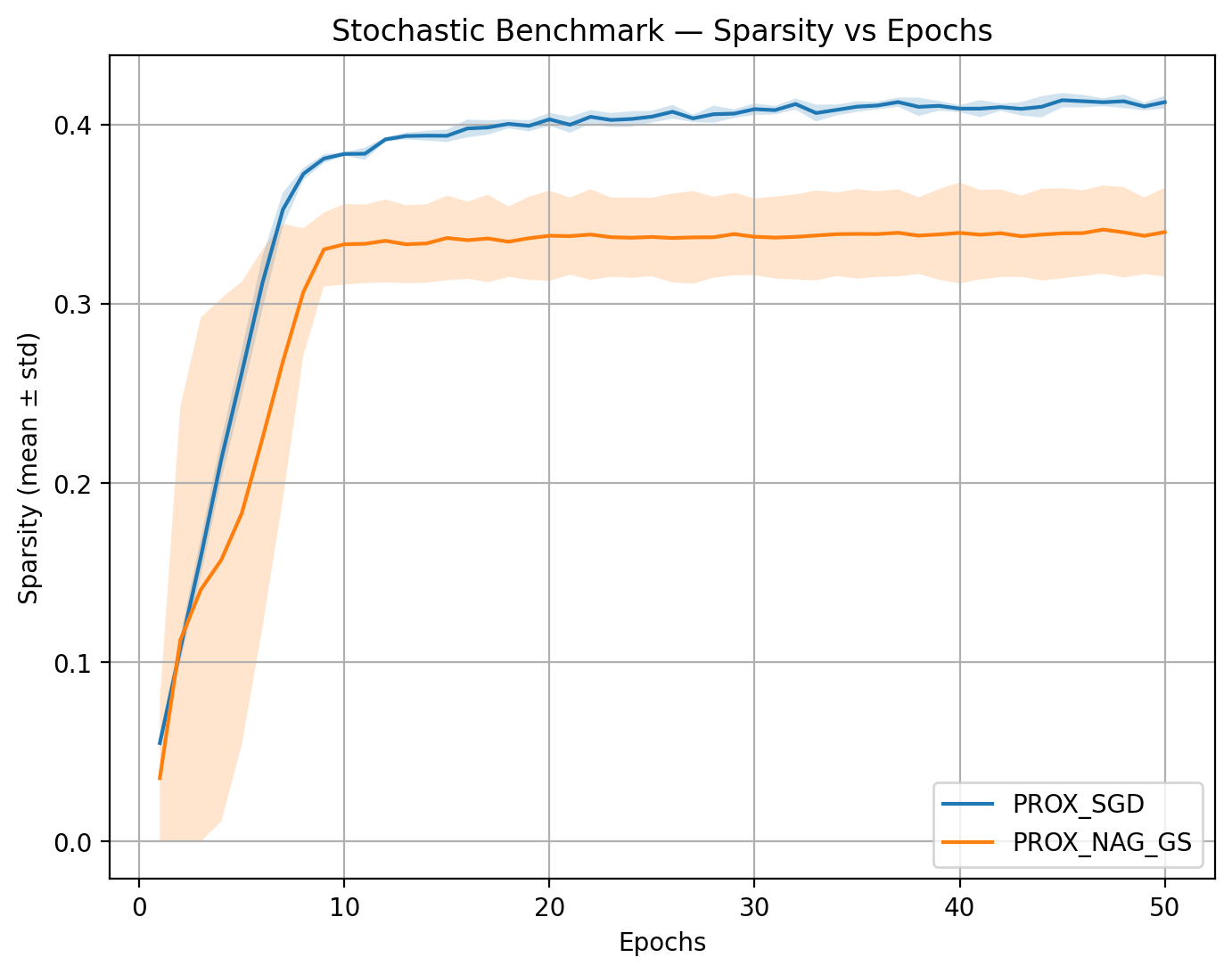}
\includegraphics[width=0.48\textwidth]{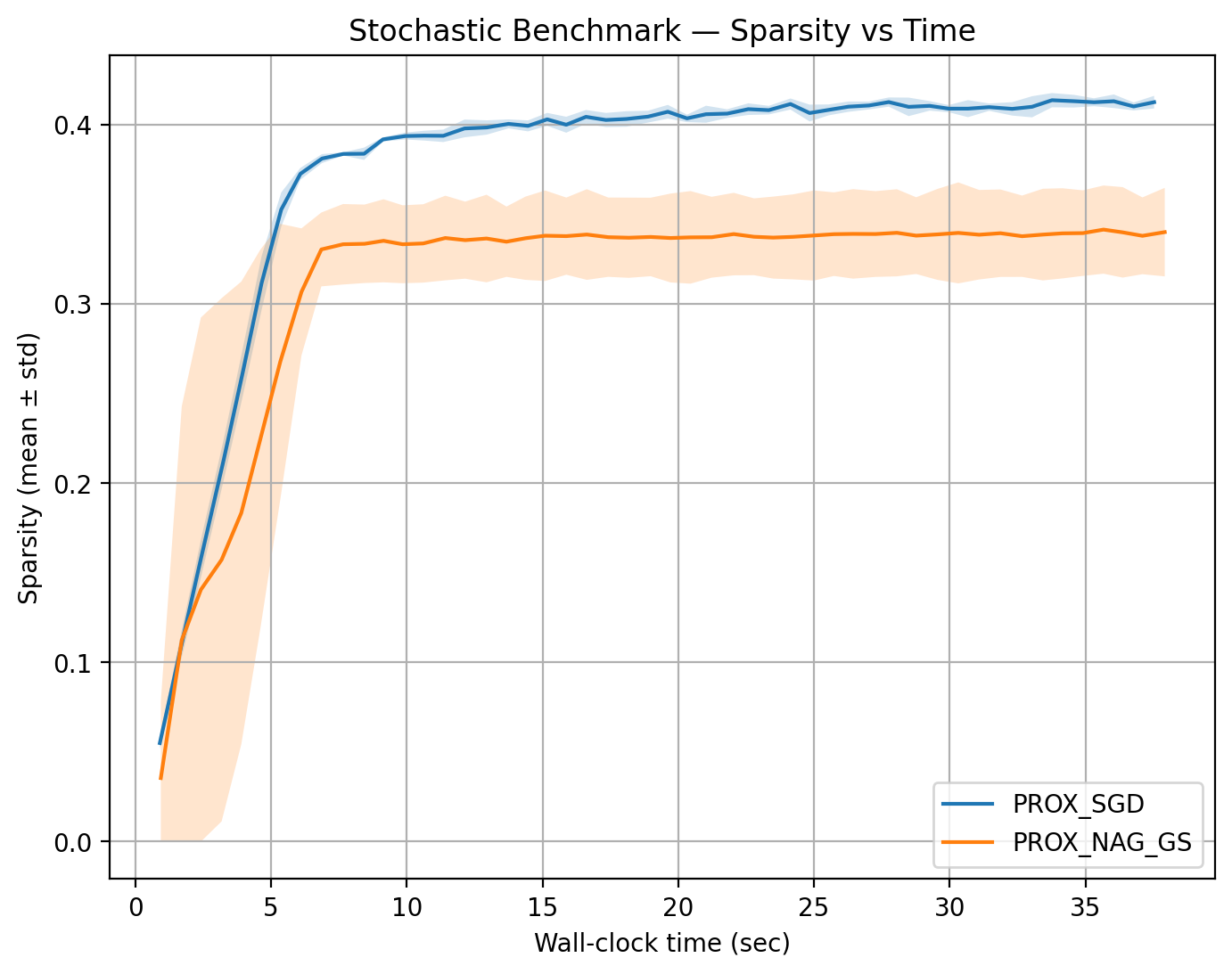}
\caption{Stochastic \(\ell_1\) softmax regression. Sparsity versus epochs and wall-clock time.}
\label{fig:stoch-l1-sparsity}
\end{figure}

\FloatBarrier

The stochastic \(\ell_1\) experiment complements the deterministic tests.
Prox-NAG-GS obtains a slightly lower full objective and a lower data-fit term for the base value
of \(\lambda_1\), with essentially the same test accuracy. However, Prox-SGD gives a smaller
regularization term and produces a sparser model.

The sweep over \(\lambda_1\) confirms this trade-off. For small and moderate regularization,
Prox-NAG-GS remains competitive in accuracy, but it produces denser solutions. For the largest
value of \(\lambda_1\), the final iterate of Prox-NAG-GS becomes less stable: the final test accuracy
drops, even though the best validation checkpoint remains competitive. This suggests that, in the
strongly regularized stochastic regime, the present version of Prox-NAG-GS may need additional
stabilization, for instance damping, averaging, or restart strategies.

\FloatBarrier

\subsection{Stochastic Group Lasso}

The stochastic Group Lasso benchmark also uses softmax regression on MNIST. We define one
group per input pixel, grouping the weights associated with this pixel across all output classes:
\begin{equation}
\label{eq:stoch-group-num}
    \min_W
    \frac1n\sum_{i=1}^n\ell_i(W)
    +\lambda_g\sum_{G\in\mathcal G}\norm{W_G}_2
    +\frac{\lambda_2}{2}\norm{W}_F^2 .
\end{equation}
The model has \(784\) groups, each of size \(10\).

\begin{table}[H]
\centering
\caption{Stochastic Group Lasso on MNIST softmax regression. Mean \(\pm\) standard deviation over five seeds.}
\label{tab:stoch-group-summary}
\resizebox{\textwidth}{!}{
\begin{tabular}{lcccccc}
\toprule
Method & Final obj. & Data-fit & Reg. term & Test acc. & Group sparsity & Time (s) \\
\midrule
Group Prox-SGD
& \(\mathbf{0.2983\pm0.0010}\)
& \(0.2668\pm0.0011\)
& \(\mathbf{0.0315\pm0.0003}\)
& \(\mathbf{0.9229\pm0.0007}\)
& \(\mathbf{0.2566\pm0.0098}\)
& \(615.83\pm11.40\) \\
Group Prox-NAG-GS
& \(0.3061\pm0.0140\)
& \(\mathbf{0.2637\pm0.0109}\)
& \(0.0424\pm0.0040\)
& \(0.9200\pm0.0066\)
& \(0.1852\pm0.0272\)
& \(\mathbf{612.42\pm13.34}\) \\
\bottomrule
\end{tabular}
}
\end{table}

\FloatBarrier

\begin{figure}[!htbp]
\centering
\includegraphics[width=0.48\textwidth]{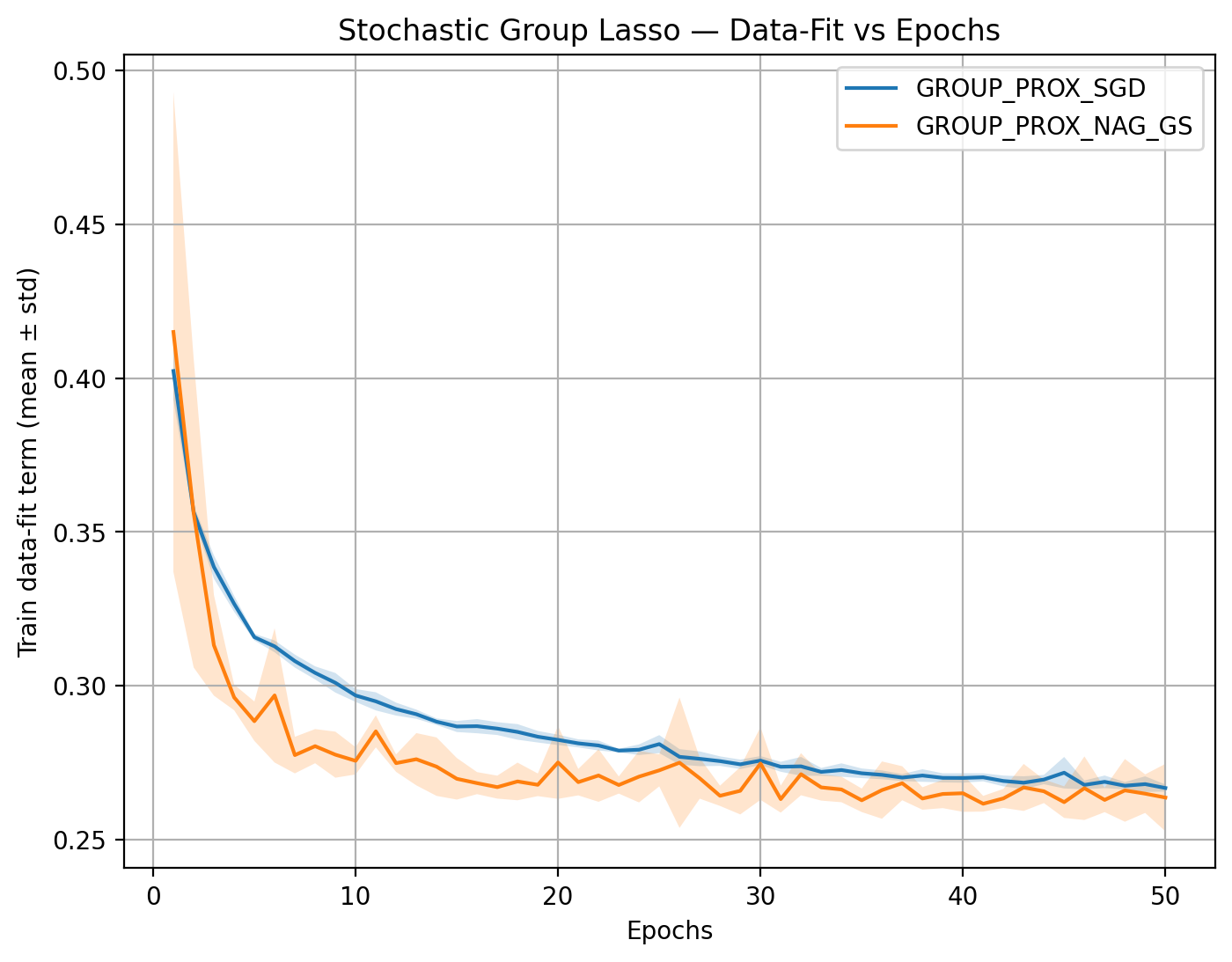}
\includegraphics[width=0.48\textwidth]{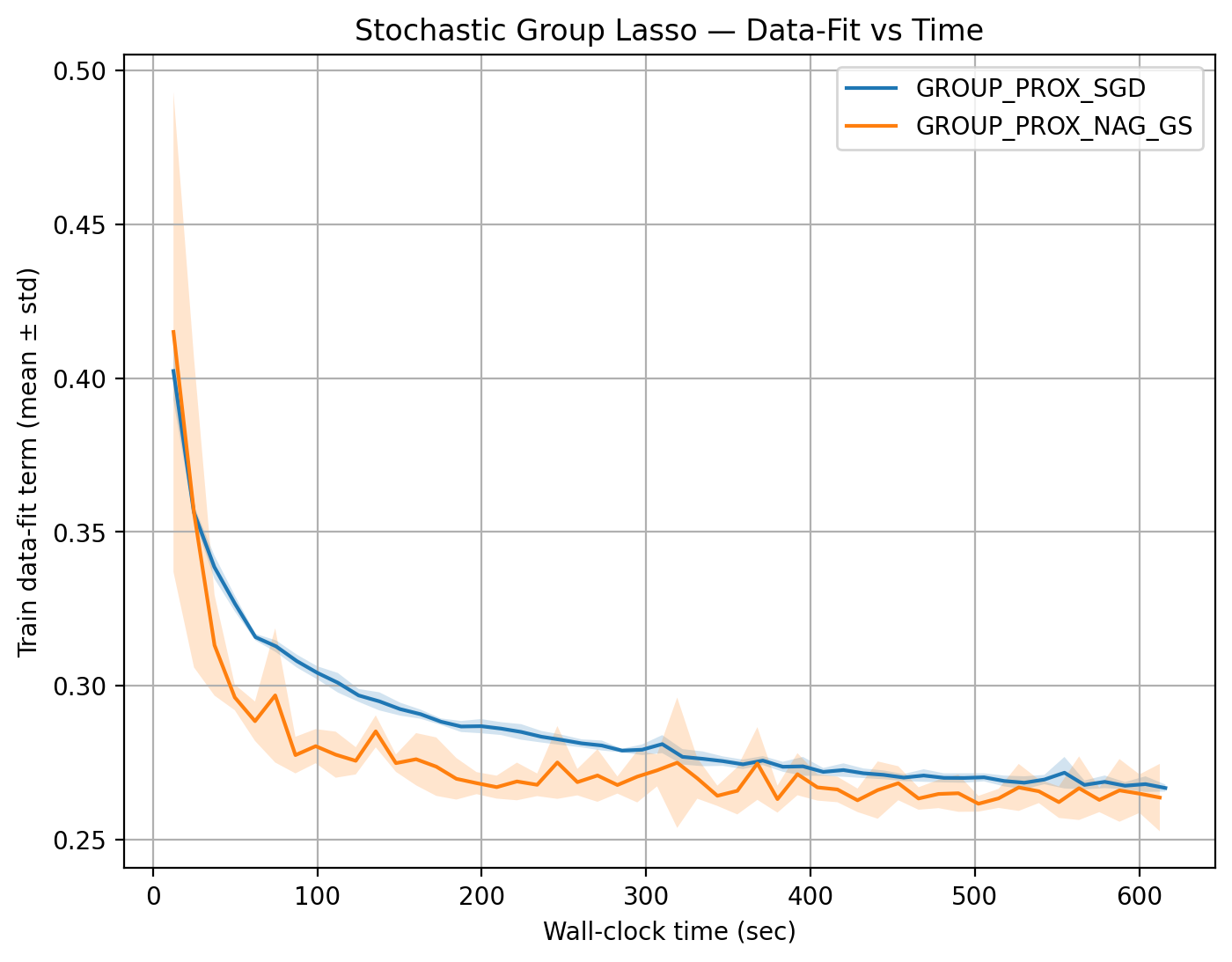}
\caption{Stochastic Group Lasso. Training data-fit term versus epochs and wall-clock time.}
\label{fig:stoch-group-datafit}
\end{figure}

\FloatBarrier

\begin{figure}[!htbp]
\centering
\includegraphics[width=0.48\textwidth]{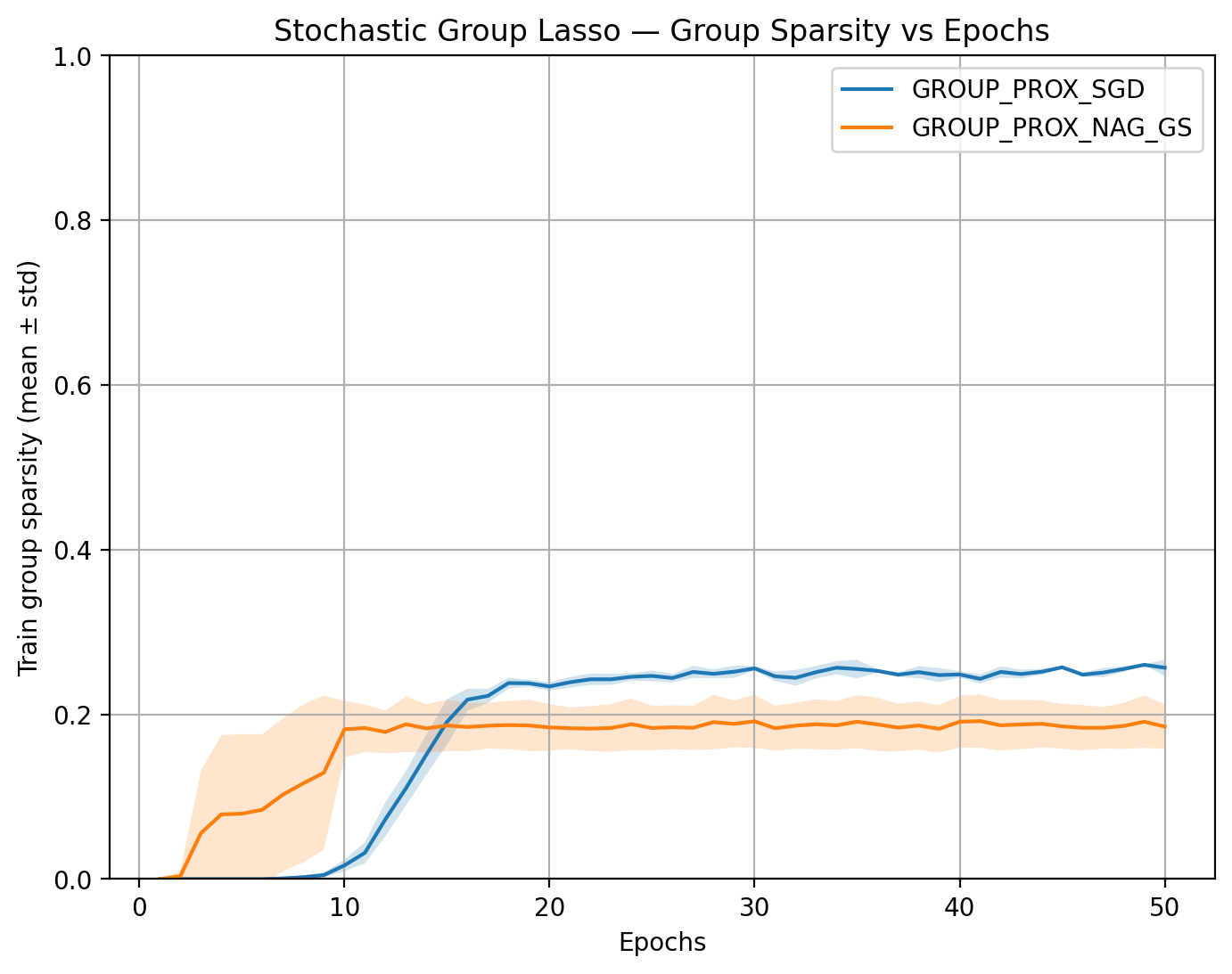}
\includegraphics[width=0.48\textwidth]{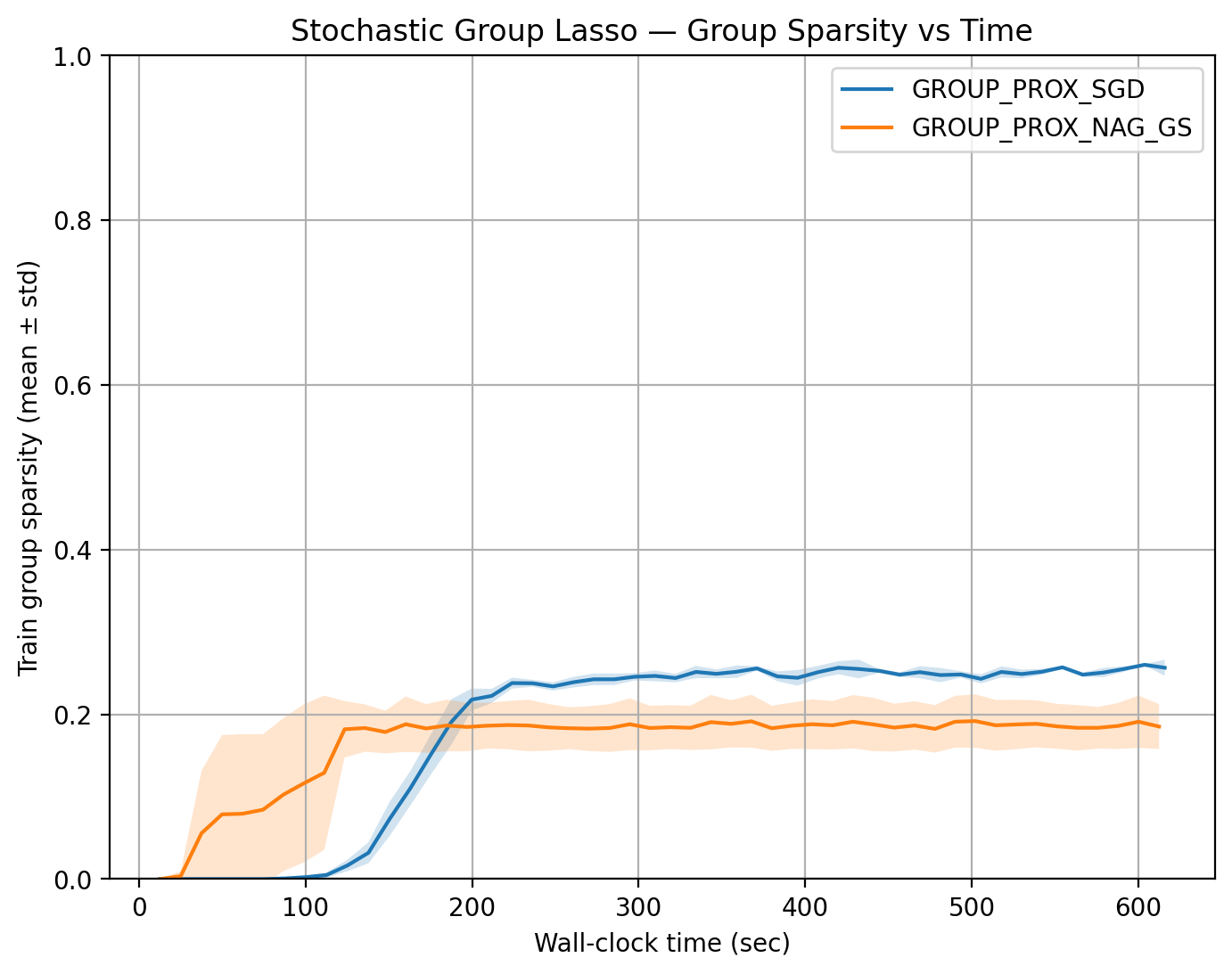}
\caption{Stochastic Group Lasso. Group sparsity versus epochs and wall-clock time.}
\label{fig:stoch-group-sparsity}
\end{figure}

\FloatBarrier

The stochastic Group Lasso experiment is consistent with the stochastic \(\ell_1\) test.
Group Prox-NAG-GS gives the lowest data-fit term, while Group Prox-SGD gives the lowest full
regularized objective. The reason is again the regularization term: Prox-SGD produces stronger
group sparsity, whereas Prox-NAG-GS keeps denser weights. The test accuracies are close, with a
small advantage for Prox-SGD at the final iterate. The wall-clock times are comparable.

\FloatBarrier

\subsection{Discussion of the experiments}

The deterministic experiments support the proximal extension clearly. On both Elastic Net and
Group Lasso, Prox-NAG-GS reaches the same objective value as the baselines and needs fewer
iterations. The strongest deterministic result is obtained on the Group Lasso benchmark, where
the iteration advantage also leads to the best wall-clock time.

The stochastic experiments also show competitive behavior. Prox-NAG-GS compares favorably
with Prox-SGD in terms of data-fit reduction and gives similar test accuracies. In both stochastic
benchmarks, the method tends to decrease the smooth empirical loss more strongly. At the same
time, the experiments reveal a different regularization behavior. Prox-SGD usually produces
sparser models and may give a lower full regularized objective when the nonsmooth regularization
is large.

Overall, the experiments suggest that the semi-implicit proximal coupling is useful in deterministic
composite problems, especially when the proximal operator is structured but still cheap to compute.
They also suggest that Prox-NAG-GS is promising in stochastic composite learning, but that its
regularization behavior may need more careful damping, averaging, or restart strategies.

\paragraph{Reproducibility and runtime.}
The full Python code, including the scripts used to regenerate the figures and tables reported in
this section, is available at
\[
\text{\url{https://github.com/giselesikeh/prox-nag-gs-composite-optimization.git}}.
\]
The main experimental parameters are specified explicitly in the corresponding scripts, including
the number of Optuna trials, the maximum number of iterations or epochs, the batch size, the
MNIST train/validation/test split, the dimensions of the deterministic problems, and the values of
\(\lambda_1,\lambda_2,\lambda_g\).
All experiments were run on Google Colab using an A100 GPU. The reported times are wall-clock
times measured in this environment. Absolute runtimes may differ on other hardware, but the
repository contains the code and settings needed to reproduce the reported comparisons.
\section{Conclusion}
\label{sec:conclusion}

We introduced Prox-NAG-GS, a proximal extension of NAG-GS for composite optimization
problems. The construction is direct: the second NAG-GS update can be written as the minimizer
of a quadratic model, and adding the nonsmooth term to this model gives a proximal update. The
method therefore keeps the semi-implicit coupling of NAG-GS while allowing nonsmooth
regularizers and simple convex constraints to be handled through their proximal operators.

We also provided deterministic convergence guarantees. The main point in the analysis is that the
gradient is evaluated at \(x_{k+1}\), while the proximal step returns \(v_{k+1}\). This creates a
mismatch that does not appear in the standard proximal-gradient proof. Under the sufficient
condition \(\widehat\mu\ge L\), this term can be controlled. In the strongly convex composite case,
we prove a linear convergence result using an augmented Lyapunov function involving both
\(\|v_k-x^\star\|^2\) and \(\|x_k-x^\star\|^2\). We also showed that the same Lyapunov structure
gives an \(O(1/k)\) rate for the best iterate and for the averaged iterate in the convex case. These
results do not claim optimal accelerated rates, but they provide a rigorous convergence analysis
for the proximal version of NAG-GS.

The deterministic experiments support the method. On Elastic Net and Group Lasso benchmarks,
Prox-NAG-GS reaches the same objective values as ISTA, FISTA and Chambolle-Pock, while
requiring fewer iterations. The Group Lasso experiment gives the strongest deterministic evidence:
Prox-NAG-GS is faster both in iterations and in wall-clock time. The stochastic experiments also
show competitive behavior. Prox-NAG-GS compares favorably with Prox-SGD in terms of data-fit
reduction and gives similar test accuracies. At the same time, the experiments reveal a different
regularization behavior: Prox-SGD usually produces sparser solutions and may give a lower full
regularized objective when the nonsmooth regularization is large. This suggests that Prox-NAG-GS
is promising in stochastic composite learning, but that its regularization behavior may need more
careful damping, averaging, or restart strategies.

Several questions remain open. The condition \(\widehat\mu\ge L\) used in the proof is conservative
and does not cover all parameter choices that work well in practice. A sharper analysis could
explain larger effective stepsizes and possibly lead to accelerated rates. Another natural direction
is the stochastic analysis of Prox-NAG-GS, especially under mini-batch gradients and variance
assumptions. Finally, broader numerical tests on larger-scale learning problems would help clarify
when the semi-implicit proximal coupling is most beneficial.

\bibliographystyle{tfs}
\bibliography{references}

\end{document}